\documentclass{amsart}                          
\usepackage{amssymb,amscd,bbm}     
\usepackage{amsthm,amsgen,amsmath,empheq}
\usepackage[all]{xy}
\usepackage{mathdots}

\xyoption{curve}                
\xyoption{rotate}               

\newcommand{\linep}[2]{ \ensuremath{               %
 \begin{xy}                                        
	(0,-2)*[o]+=<7pt>{\scriptstyle #1}*\frm{o}="a",
	(4,2)*[o]+=<7pt>{\scriptstyle #2}*\frm{o}="b", 
  "a";"b"**\dir{-}?>*\dir{>},                      
 \end{xy}                                          %
} }
\newcommand{\clinep}[2]{ \ensuremath{             %
 \begin{xy}                                       
	 (0,-2)*[o]+U{\scriptstyle #1}*+\frm{.o}="a", 
	 (6,2)*[o]+U{\scriptstyle #2}*+\frm{.o}="b",  
  "a";"b"**\dir{-}?>*\dir{>},                     
 \end{xy}                                         %
} }
\newcommand{\cllinep}[2]{ \ensuremath{            %
 \begin{xy}                                       
	 (0,-2)*[o]{\scriptstyle #1}*+\frm{.e}="a",   
	 (14,2)*[o]{\scriptstyle #2}*+\frm{.e}="b",   
  "a";"b"**\dir{-}?>*\dir{>},                     
 \end{xy}                                         
} }

\newcommand{\clinem}[2]{ \ensuremath{             %
 \begin{xy}                                       
	 (0,-2)*[o]+U{\scriptstyle #1}*+\frm{.o}="a", 
	 (6,2)*[o]+U{\scriptstyle #2}*+\frm{.o}="b",  
  "b";"a"**\dir{-}?>*\dir{>},                     
 \end{xy}                                         %
} }
\newcommand{\cllinem}[2]{ \ensuremath{            %
 \begin{xy}                                       
	 (0,-2)*[o]{\scriptstyle #1}*+\frm{.e}="a",   
	 (14,2)*[o]{\scriptstyle #2}*+\frm{.e}="b",   
  "b";"a"**\dir{-}?>*\dir{>},                     
 \end{xy}                                         %
} }

\newcommand\upx[3]{\begin{xy}(0,#3)*{#2};(0,3)*{#1}\end{xy}}

\newcommand{\T}{ \ensuremath{      
	\mathbb{T}
} }

\newcommand{\Lie}{ \ensuremath{    
	\mathbb{L}
} }

\newcommand{\R}{ \ensuremath{      
	\mathbb{P}
} }

\newcommand{\G}{ \ensuremath{      
	\mathbb{G}
} }

\newcommand{\Aop}{ \ensuremath{    
	\mathcal{A}\!\mathit{s}
} }

\newcommand{\Lop}{ \ensuremath{    
	\mathcal{L}\!\mathit{ie}
} }

\newcommand{\Rop}{ \ensuremath{    
	\mathit{pre}\!\mathcal{L}\!\mathit{ie}
} }

\newcommand{\Gop}{ \ensuremath{    
	\mathcal{G}\!\mathit{r}
} }

\newcommand{\Op}{ \ensuremath{     
  \mathcal{O}
} }

\newtheorem{theorem}{Theorem}[section]
\newtheorem{proposition}[theorem]{Proposition}                          
\newtheorem{lemma}[theorem]{Lemma}

\theoremstyle{definition}
\newtheorem{definition}[theorem]{Definition}
\newtheorem{corollary}[theorem]{Corollary}                          
    
\newtheorem{example}[theorem]{Example}

\theoremstyle{remark}
\newtheorem{remark}[theorem]{Remark}

\numberwithin{equation}{section}
                     
\theoremstyle{definition}                          

\theoremstyle{remark}                          
\newtheorem*{notation}{Notation}

\begin{document}

\title{Lie Algebra Configuration Pairing}

\author[B. Walter]{Ben Walter} 
\address{
Department of Mathematics \\ Middle East Technical University, Northern Cyprus Campus \\
Kalkanli, Guzelyurt, KKTC, Mersin 10 Turkey
}
\email{benjamin@metu.edu.tr}



\begin{abstract}
 	We give an algebraic construction of the topological graph-tree configuration pairing
    of Sinha and Walter beginning with the classical presentation of Lie
    coalgebras via coefficients of words in the associative Lie polynomial.
	Our work moves from associative algebras to preLie algebras to the graph 
	complexes of Sinha and Walter, 
    justifying the use of 
    graph generators for Lie coalgebras by iteratively expanding the set of 
    generators until the set of relations collapses to two simple local expressions.
    Our focus is on new computational methods allowed by this framework and
    the efficiency of the graph presentation in proofs and calculus involving
    free Lie algebras and coalgebras.
	This outlines a new way of understanding and calculating with Lie algebras
    arising from the graph presentation of Lie coalgebras. 
\end{abstract}

\maketitle

%
%

\section*{Introduction}

The configuration pairing of graphs and trees has its genesis in 
\cite{Sinh06} as an explicit
geometric description of the homology/cohomology
pairing for configuration spaces and the disks operad.  
Cycles in configuration space homology
are realized by submanifolds, where points in configurations orbit each other
organized into systems and galaxies.  Cohomology cocycles check whether certain 
arrangements of points can ever occur in a homology galaxy.  Algebraically,
homology galaxies are encoded as trees and cohomology cocycles are
written as directed graphs.  
Anti-symmetry
and Jacobi expressions of trees bound, and so vanish in homology; dually 
arrow-reversing and Arnold expressions of graphs cobound in cohomology.
The homology of configuration spaces is the
Poisson operad (which can be expressed as forests of trees),
so cohomology gives a presentation of the Poisson cooperad (expressed with graphs).  
Restricting to connected objects, this duality descends to
an equivalence of Lie coalgebras and directed graphs modulo
arrow-reversing and Arnold.  This is exploited in \cite{SiWa07}, \cite{SiWa09},
and \cite{Walt10}.  For more information on how to construct the graph cooperad
without reference to operads, the interested reader may consult \cite{Walt13}.

In this paper we give an alternate view of the configuration pairing 
between Lie algebras and graphs, grounded
not in topology but algebra.  There is a commutative diagram
of functors of categories (described via operad maps in Appendix~\ref{A:operad}).
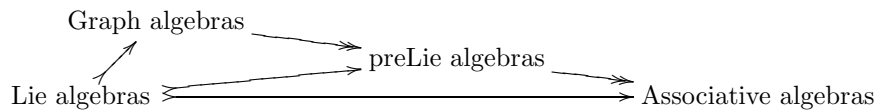
\begin{equation*}
\begin{aligned}
\begin{xy}
	(0,0)*+<20pt,10pt>{\text{Lie algebras}}="l",
	(90,0)*+{\text{Associative algebras}}="a",
	(50,5)*+{\text{preLie algebras}}="r",
	(10,10)*+{\text{Graph algebras}}="g",
	"l";"a"**\dir{-}?>*\dir{>}?<*\dir{>},
	"l";"r"**\dir{-}?>*\dir{>}?<*\dir{>},
	"l";"g"**\dir{-}?>*\dir{>}?<*\dir{>},
	"g";"r"**\dir{-}?>*\dir{>>},
	"r";"a"**\dir{-}?>*\dir{>>}
\end{xy}
\end{aligned}
\end{equation*}
The map from Lie algebras to associative algebras is the 
universal enveloping algebra map. 
By analogy, we call the other 
maps from Lie algebras also ``universal enveloping'' maps --
they are all adjoint to functors forgetting
structure in the other direction and are defined similar to the
universal enveloping (associative) algebra map. 
The maps marked $\twoheadrightarrow$ are quotient maps on algebras.
Up to slight tweaks to coalgebra structure, we construct a dual 
diagram in coalgebras.

The dual maps to Lie coalgebras yield presentations of Lie coalgebras
as quotients of associative, preLie, and graph coalgebras.
The duality of Lie algebras with the associative coalgebra presentation
of Lie coalgebras
is given by computing coefficients in the associative polynomial of a
Lie expression.  
The duality of Lie algebras with the graph presentation of Lie coalgebras 
is the topological generalized Hopf invariant pairing of
\cite{SiWa09}.
The preLie presentation
interpolates between these two.

As an example of the power of the graph framework, we use it to 
gain a trivial proof of a deep description of Lie coalgebras
(which, while not explicitly stated previously in the literature, 
 could be inferred from~\cite{MeRe89} or~\cite{Reut93}).
Corollary~\ref{C:ass L^*} states that free Lie coalgebras 
are no more than anti-commutative coalgebras
with the correct set of primitives.  
At the end of Section~\ref{S:graphs}, we also give a new 
method for computing coefficients in the associative polynomial 
of a Lie bracket
(Proposition~\ref{P:ass combinatorial def}) induced by the 
combinatorial definition of the configuration pairing of graphs and trees.
In later work we will give 
further applications computing in Lie algebras 
and shuffle algebras. 

\medskip

%
%

The paper is organized as follows.  
Section~\ref{S:ass} recalls the classical situation of Lie 
algebras and universal enveloping associative algebras.  We arrange ideas and 
notation anticipating our later sections.  
In the context of associative algebras, the configuration pairing
of an associative word and Lie bracket expression is the coefficient of the word in 
the associative (noncommutative) polynomial of the Lie bracket expression.
Material in this section is 
all classical aside possibly from the presence of Lie and associative coalgebra 
structures in Proposition~\ref{P:ass compat} and Corollary~\ref{C:ass L^*}.
The use of coalgebra structures in Lie polynomial coefficient computations 
in Example~\ref{E:comp} may also be new, though it is not deep.

In Section~\ref{S:prelie} we work analogous to Section~\ref{S:ass}, but 
with preLie structures.  We spend more time on preliminaries since we
expect this to be less familiar territory.  Our construction of Lie 
coalgebra structures from preLie structures is similar in motivation to 
\cite{Mark06}, though we work with duals and use the configuration 
pairing.
In this setting we develop two views of the configuration pairing.
The algebraic configuration pairing is defined similar to Section~\ref{S:ass},
as reading the coefficient of a preLie element in the formal preLie polynomial of
a Lie bracket expression.  The combinatorial configuration pairing is defined
in terms of vertex-labeled rooted preLie trees in the spirit of \cite{SiWa07}. 

Section~\ref{S:graphs} is motivated by the presentation of Lie coalgebras
via preLie coalgebras given at the end of Section~\ref{S:prelie}.
A theory analogous to the previous sections holds for graph algebras. 
Moving from preLie to graphs enlarges the number of generators in the resulting
presentation of Lie coalgebras, but greatly simplifies the set of relations.
A simpler presentation of Lie coalgebras leads to greatly streamlined proofs.
In fact, we leave the final step of the proofs of a few propositions in
sections~\ref{S:ass} and \ref{S:prelie}
until section~\ref{S:graphs}, when they become simple observations of  
structure induced from graph algebras via quotient maps 
to preLie algebras and associative algebras.

In the Appendix~\ref{A:operad} we give a few operad-level constructions 
and proofs.  We also give 
a description of the full graph algebra and coalgebra structure, which 
we omit from Section~\ref{S:graphs} for simplicity. 
Appendix~\ref{A:coalg} gives short proofs of a few basic coalgebra lemmas
for the benefit of neophyte readers.

Throughout, we will assume that our algebras have underlying 
$k$-vector spaces.  
In particular, we make frequent use of the free algebra
maps from $k$-vector spaces to algebras.  
For brevity, we write $\otimes$ for $\otimes_k$.
In remarks, we discus interpretations
of definitions and propositions, given a chosen basis 
$B = \{b_i\}_{i\in I}$ of a $k$-vector 
space $V$. 
Furthermore, in the interest of all constructions naturally 
connecting to Lie algebras, all of our algebras and coalgebras 
will be without unit and counit.

\medskip

The genesis of this paper was notes of a talk concieved, developed,
and given all during the course of the week-long Workshop on Operads in 
Homotopy Theory (2010) in Lille.  In the intervening time many elements have 
been cleaned and refined.

%
%

\section{The Configuration Pairing with Associative Coalgebras}\label{S:ass}

We recall the classical theory of Lie algebras and their universal enveloping
algebras, setting notation for later sections, 
and carefully developing the linear dual of the universal enveloping 
algebra map (to avoid concerns about infinite dimensional coalgebra 
structures).  

Given a $k$-vector space $V$,
$\T V=\bigoplus_{n\ge 1} V^{\otimes n}$ is the free nonunital associative 
algebra on $V$.  Write the word $x_1\cdots x_n$ for the homogeneous element 
$x_1\otimes\cdots\otimes x_n\in \T V$. 
The universal enveloping algebra of a Lie algebra is 
$U_AL = \T L/\!\!\!\sim$, 
where $[x,y]\sim xy-yx$.  
Given a Lie algebra $L$, write $p_A:L\to U_A L$ for the composition
$L\rightarrowtail \T L\twoheadrightarrow U_AL$.
The map $p_A$ sends Lie elements to their the associative 
Lie polynomials. 
Note that $p_A$ is not an algebra map unless we twist $\T V$ to have the
anti-commutative product: $\mu(x, y) = xy - yx$.
Write $\Lie V$ for the free Lie algebra on $V$ and recall the 
classical isomorphism $U_A\Lie V \cong \T V$. 
We are interested in the map $p_{A}:\Lie V\to U_A\Lie V \cong \T V$ and its dual.
Let $(\T V)^*$ and $(\Lie V)^*$ be the vector space duals of $\T V$ and $\Lie V$.

\begin{remark}
	In order to have honest coalgebra structures on duals, we must enforce 
    a finiteness condition when dualizing to $(\T V)^*$.
    For example, we could restrict all functionals $\psi \in (\T V)^*$ 
    to have finite dimensional support.
	Alternately we could weaken the definition of coalgebra to allow formal 
	sums in the coproduct operation.  The coproduct operation then lands in
	the completed tensor product $\Delta:C \to C\,\widehat{\otimes}\;C$ and
	the coalgebra axioms are all modified accordingly.

	Throughout this paper, all constructions will be grounded via pairings 
	with algebras.  Due to finiteness conditions on the algebra side, 
	these pairings will never involve infinite sums of nonzero
	elements, no matter which alternative for a specific 
	construction / definition of dual is adopted.  Since our goal is 
	an understanding of algebras via our explicit pairings, we 
	structure our statements and proofs to sidestep issues related to 
    infinite dimensional coalgebras.  
\end{remark}

\begin{notation}
    We write $\psi \in (\T V)^*$ for a generic functional on $\T V$ 
    [with finite dimensional support].  We use $\omega$ for 
    homogeneous elements of $\T V$ (words $x_1\cdots x_n$) and 
    $\omega^* \in (\T V)^*$
    for dual functional of the word $\omega$. 
    Homogeneous elements of $(\T V)^*$ are the elements of the form $\omega^*$
    for some homogeneous $\omega \in \T V$.
    We write $\ell\in\Lie V$ for generic Lie bracket expressions.
\end{notation}

\begin{definition}\label{D:ass}
	Define the vector space pairing 
	$\langle -,-\rangle:(\T V)^*\otimes \Lie V \to k$ to be 
	$\langle \psi,\,\ell\rangle = \psi\bigl(p_A(\ell)\bigr)$.

	Let $\eta_A: (\T V)^* \to (\Lie V)^*$ be the associated map $\psi\mapsto \langle \psi,-\rangle$.
\end{definition}

\begin{remark}
	The map $\eta_A$ is the dual of $p_A$ as a map of 
	vector spaces.

	If $V$ has a chosen basis $B$, then we may canonically write 
    elements of $\T V$ as linear combinations of words in the alphabet $B$.
    Thus the associative Lie polynomial for a Lie bracket becomes 
    $p_A(\ell) = \sum_i c_i\omega_i$ where the $\omega_i$ are words in $B$. 
	In this case $\langle \omega_i^*,\, \ell\rangle = \omega_i^*(p_A(\ell)) = c_i$,  
	the coefficient of the word $\omega_i$ in the associative Lie polynomial $p_A(\ell)$.
    More generally $\eta_A(\omega^*) \in (\Lie V)^*$ is the functional which 
    reads off the $\omega$ coefficient of associative Lie polynomials.
 \end{remark}

The usual product on $\T V$ induces via duality
a coalgebra structure on $(\T V)^*$ which operates on 
homogeneous elements by cutting a word at 
all possible positions (yielding nonempty subwords) 
$\Delta \omega^* = \displaystyle \sum_{\omega=\alpha\beta} \alpha^*\otimes \beta^*$.
Due to the definition of $p_A$, the map $\eta_A$ will be a map of 
coalgebras only after twisting the coalgebra structure of $(\T V)^*$ to be
anti-commutative.  Define the cobracket of $\psi \in (\T V)^*$ to be\ 
$]\psi[\ = \Delta\psi - \tau\Delta\psi$ where $\tau$ is the twist map.
On homogeneous elements the cobracket is 
$\displaystyle ]\,\omega^*\![\ = \sum_{\omega=\alpha\beta} 
   (\alpha^*\otimes \beta^* - \beta^*\otimes \alpha^*)$.

\begin{proposition}\label{P:ass compat}
	If $\psi \in (\T V)^*$ and $\ell_1,\ell_2\in \Lie V$, then
	$$\bigl\langle \psi,\, [\ell_1, \ell_2]\bigr\rangle \ = \ 
	\sum_{i} \langle \alpha_i,\, \ell_1\rangle\,\langle \beta_i,\,\ell_2\rangle,$$
	where $]\psi[\ = \sum_i \alpha_i\otimes \beta_i$.
\end{proposition}
\begin{proof}
	This follows immediately from 
	$p_A([\ell_1,\ell_2]) = p_A(\ell_1)p_A(\ell_2) - p_A(\ell_2)p_A(\ell_1)$ for
	homogeneous Lie bracket expressions $\ell_1$ and $\ell_2$ and the definition
    of the pairing $\langle-,-\rangle$.
\end{proof}

By Proposition~\ref{P:ass compat} the pairing $\langle-,\,-\rangle$ becomes a  
coalgebra/algebra pairing if $(\T V)^*$ is given the cobracket coalgebra structure.
From now on, we will always equip $(\T V)^*$ with the cobracket coalgebra structure.

\begin{remark}
Proposition~\ref{P:ass compat} gives a method for recursive calculation of
coefficients for Lie bracket expressions $p_A([\ell_1,\,\ell_2])$ using 
cobrackets of words.  
The coefficient of a word $\omega$ in the 
associative polynomial $p_A\bigl([\ell_1,\ell_2]\bigr)$
is given by 
\[ \Bigl\langle \omega^*,\, [\ell_1, \ell_2]\Bigr\rangle = 
   \sum_{\omega=\alpha\beta} 
               \langle \alpha^*,\, \ell_1\rangle\, \langle \beta^*,\, \ell_2\rangle 
             - \langle \beta^*,\, \ell_1\rangle\, \langle \alpha^*,\, \ell_2\rangle\]
Note that the sum above involves at most two nonzero terms, 
since $\langle \omega^*,\ell\rangle = 0$ unless the lengths 
of $\omega$ and $\ell$ match.
    
\begin{example}\label{E:comp}
  For example
  the coefficient of $abbba$ in $p_A([[[b,a],b],[a,b]])$ may be computed 
  as follows.
  \begin{align*}
	\bigl\langle abbba^*,\, [[[b,a],b],[a,b]]\bigr\rangle &= 
	\bigl\langle abb^*,\,[[b,a],b]\bigr\rangle\,\bigl\langle ba^*,[a,b]\bigr\rangle\ \\ 
       &\quad -\ 
	  \bigl\langle bba^*,\,[[b,a],b]\bigr\rangle\,\bigl\langle ab^*,[a,b]\bigr\rangle \\
  &= \Bigl(\bigl\langle ab^*,\,[b,a]\bigr\rangle\,\langle b^*,\,b\rangle\ -\
    \bigl\langle bb^*,\,[b,a]\bigr\rangle\,\langle a^*,\,b\rangle\Bigr)(-1)  \\
      &\quad -\ 
    \Bigl(\bigl\langle bb^*,\,[b,a]\bigr\rangle\,\langle a^*,\,b\rangle\ -\ 
    \bigl\langle ba^*,\,[b,a]\bigr\rangle\,\langle b^*,\,b\rangle\Bigr)(1)\\
  &= \bigl((-1)(1) + (0)(0)\bigr)(-1) \\
    &\quad - \bigl((0)(0) - (1)(1)\bigr)(1) = 2.
	\end{align*}
\end{example}

	An alternate method of computing $\langle-,-\rangle$ will follow from our
	work in Section~\ref{S:prelie} 
	(see Proposition~\ref{P:prelie geom conf} and Example~\ref{E:alt comp}).
    Another, more general method follows from work in Section~\ref{S:graphs}
    (see Proposition~\ref{P:ass combinatorial def} and 
    Example~\ref{E:ass combinatorial def}).
\end{remark}

The map $p_A$ to universal (associative) enveloping algebras is an injection
by a simple corollary of the Poincar\'e-Birkhoff-Witt theorem.  Thus
Proposition~\ref{P:ass compat} has the following corollary.

\begin{corollary}\label{C:ass surj}
	The surjection $\eta_A:(\T V)^* \to (\Lie V)^*$ is a coalgebra homomorphism. 
    Thus $(\Lie V)^* \ \cong \ (\T V)^*\,/\ \ker(\eta_A)$.
\end{corollary}

\begin{remark}
 Note that $\psi\in\ker(\eta_A)$ if and only if $\langle\psi,\,\ell\rangle = 0$ for
 all $\ell \in \Lie V$.
\end{remark}

Let $\bigl\langle\ker(\,]\cdot[\,)\bigr\rangle \subset (\T V)^*$ be the smallest coideal of
$(\T V)^*$ containing \mbox{$\ker(\,]\!\cdot\![\,)\setminus V^*$}.   
The following proposition is suggested by various classical results; we will give a new, 
simple proof of it later.
(Proposition~\ref{P:ass ker =}
follows immediately from Proposition~\ref{P:prelie ker =}; Proposition~\ref{P:prelie ker =}
is a direct consequence 
of Proposition~\ref{P:graphs ker =}; Proposition~\ref{P:graphs ker =} is simple to prove.)
The corollary says that Lie coalgebras are universal anti-commutative coalgebras with
nontrivial coproducts.

\begin{proposition}\label{P:ass ker =}
	$\ker(\eta_A) = \bigl\langle\ker(\,]\cdot[\,)\bigr\rangle$. 
\end{proposition}

\begin{corollary}\label{C:ass L^*}
	$(\Lie V)^* \ \cong \ (\T V)^*\,/\;\bigl\langle\ker(\,]\cdot[\,)\bigr\rangle$ as coalgebras.
\end{corollary}

Note that $(\T V)^*$ with anti-commutative coproduct cannot be the cofree anti-commutative
coalgebra on $V^*$, because it has too many primitives.  The coideal 
$\bigl\langle\ker(\,]\cdot[\,)\bigr\rangle$ consists of all of the elements which must
be removed to eliminate the ``primitives'' other than $V^*$.

\begin{remark}
	The idea that $(\Lie V)^* \cong (\T V)^*/\bigl\langle \ker(\,]\cdot[\,)\bigr\rangle$ 
	is already present in the first 
	section of \cite{ScSt85} (and probably elsewhere in the literature as well), 
	developed using Hopf algebra structures, dual
	to classical work of Quillen \cite{Quil69}.
	The idea that $(\Lie V)^* \cong (\T V)^*/\ker(\eta_A)$ (as vector spaces) 
	is contained in the classical approach to 
	Lie algebras via Lie (or Hall) polynomials -- for example \cite[\S 4.2]{Reut93}.
\end{remark}

To use this presentation, it remains to describe $\ker(\eta_A)$ 
explicitly.  
Classically $\bigl\langle\ker(\,]\cdot[\,)\bigr\rangle$ 
is the vector subspace spanned by shuffles~\cite{MeRe89}.
The shuffle of two words is defined recursively by 
$\mathrm{Sh}(a,b) = ab + ba$ for $a,b$ single letters 
and 
$\mathrm{Sh}(\omega, a\upsilon) = \mathrm{Sh}(a\upsilon, \omega) = 
	\bigl(a\,\mathrm{Sh}(\upsilon, \omega) + \mathrm{Sh}(\upsilon, \omega)\, a\bigr)$ for 
$\upsilon,\omega$ generic words.
Recall that the (associative) Lyndon-Shirshov words in an ordered alaphabet
\cite[\S 5]{Reut93} \cite{Chib06} are the words which are lexicographically less 
than each of their cyclic permutations.
By \cite{Radf79} the Lyndon-Shirshov words are a multiplicative basis
for the shuffle monoid.  Thus if we choose an ordered basis for $V$, then 
the Lyndon-Shirshov words in that basis are a vector space basis 
for $(\T V)^*/\ker(\eta_A)$.  

\medskip

There are two ways to improve the presentation of $(\Lie V)^*$ given above. 
The first is to move away from associative algebras, since they are 
often not a convenient location for constructive proofs
(see for example \cite[Prop. 22.8]{FHT01} compared to \cite[Lemma 2.15]{SiWa09}).
The second is to find a description of $\ker(\eta)$ not involving shuffles, 
since their span is rather complicated.  For example, applying Proposition~\ref{P:ass compat}
the expressions
$(abcde - edcba)$  and 
$(abcde + bcdea + cdeab + deabc + eabcd)$ 
are in $\ker(\eta_A)$, 
though they are far from
being shuffles; neither is it immediately apparent how to write them as sums 
of shuffles (see also Example~\ref{E:reverse}).  
Furthermore, using Corollary~\ref{C:ass L^*} in 
order to make a construction on $(\Lie V)^*$ involves making a construction 
on $(\T V)^*$ and then showing it is invariant under shuffles.  
The invariance step can be daunting.

Other descriptions of the $\ker(\eta_A)$ follow from our later work
in Sections~\ref{S:prelie} and \ref{S:graphs}, and are constructed 
directly from the map $\eta_A$ using the configuration pairing 
rather than via the cobracket.

%
%

\section{The Configuration Pairing with PreLie Colgebras}\label{S:prelie}

\subsection{PreLie algebras}

\begin{definition}\label{D:prelie algebra}
	A preLie algebra \cite{Gers63}
	is $(P,\,\triangleleft)$ where $\triangleleft:P\otimes P \to P$
	satisfies
	\begin{equation*}
	 (x\triangleleft y)\triangleleft z - x\triangleleft (y\triangleleft z) = 
	  (x\triangleleft z)\triangleleft y - x\triangleleft(z\triangleleft y).
	\end{equation*}
\end{definition}

The name ``preLie'' comes from the fact that 
 $[x,y] = x\triangleleft y - y\triangleleft x$ is a Lie bracket.
Note that all associative algebras are preLie algebras using the  
preLie product \(x\triangleleft y = x y\).

Free preLie algebras have a simple combinatorial model \cite{ChLi00}.
Given a vector space $V$,
the free preLie algebra on $V$ is the vector space of 
rooted (nonplanar) trees 
with vertices decorated by elements of $V$, modulo multilinearity.
The algebra structure $r\triangleleft t$ is given on homogeneous elements 
by summing over all possible
ways to connect the root of the rooted tree $t$ to any one of the vertices of $r$. 
It is straightforward to show that this satisfies Definition~\ref{D:prelie algebra}.
Write $\R V$ for the free preLie algebra on $V$, which we view as the 
vector space of multilinear vertex-decorated, rooted trees.

\begin{example}
	Below we give the preLie operation $\triangleleft$ combining two rooted trees.
	We indicate the root of a tree by writing it as the unique lowest vertex.
	$$
			
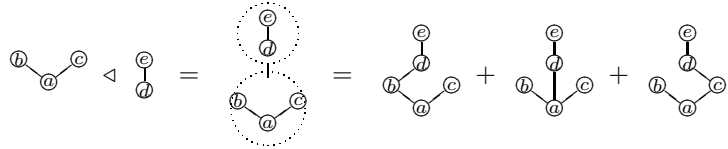
\begin{aligned}\begin{xy}
				(0,0)*[o]+=<7pt>{\scriptstyle a}*\frm{o}="1";
				(-4,3)*[o]+=<7pt>{\scriptstyle b}*\frm{o}
				**\dir{-},
				(4,3)*[o]+=<7pt>{\scriptstyle c}*\frm{o};
				"1" **\dir{-}
			\end{xy}\end{aligned}
			\ \triangleleft \ 
			\begin{aligned}\begin{xy}
				(0,-2)*[o]+=<7pt>{\scriptstyle d}*\frm{o};
				(0,2)*[o]+=<7pt>{\scriptstyle e}*\frm{o}		
				**\dir{-}
			\end{xy}\end{aligned} 
			\ \ = \ \  
			\begin{aligned}\begin{xy}
				(0,0)*[o]+=<7pt>{\scriptstyle a}*\frm{o}="a";
				(-4,3)*[o]+=<7pt>{\scriptstyle b}*\frm{o};
				**\dir{-},
				(4,3)*[o]+=<7pt>{\scriptstyle c}*\frm{o};
				"a" **\dir{-},
				(0,10)*[o]+=<7pt>{\scriptstyle d}*\frm{o};
				(0,14)*[o]+=<7pt>{\scriptstyle e}*\frm{o};
				**\dir{-},
				(0,2)*[o]+\frm<16pt,12pt>{.e},
				(0,12)*[o]+\frm<12pt,11pt>{.e},
				(0,6.2);(0,8)**\dir{-},
			\end{xy}\end{aligned}
			\ \ = \ \  
			\begin{aligned}\begin{xy}
				(0,0)*[o]+=<7pt>{\scriptstyle a}*\frm{o}="1";
				(-4,3)*[o]+=<7pt>{\scriptstyle b}*\frm{o}="2";
				**\dir{-},
				(4,3)*[o]+=<7pt>{\scriptstyle c}*\frm{o};
				"1" **\dir{-},
				(0,6)*[o]+=<7pt>{\scriptstyle d}*\frm{o}="3";
				"2" **\dir{-},
				(0,10)*[o]+=<7pt>{\scriptstyle e}*\frm{o};
				"3" **\dir{-}
			\end{xy}\end{aligned} \ + \ 
			\begin{aligned}\begin{xy}
				(0,0)*[o]+=<7pt>{\scriptstyle a}*\frm{o}="1";
				(-4,3)*[o]+=<7pt>{\scriptstyle b}*\frm{o}="2";
				**\dir{-},
				(4,3)*[o]+=<7pt>{\scriptstyle c}*\frm{o};
				"1" **\dir{-},
				(0,6)*[o]+=<7pt>{\scriptstyle d}*\frm{o}="3";
				"1" **\dir{-},
				(0,10)*[o]+=<7pt>{\scriptstyle e}*\frm{o};
				"3" **\dir{-}
			\end{xy}\end{aligned} \ + \
			\begin{aligned}\begin{xy}
				(0,0)*[o]+=<7pt>{\scriptstyle a}*\frm{o}="1";
				(-4,3)*[o]+=<7pt>{\scriptstyle b}*\frm{o};
				**\dir{-},
				(4,3)*[o]+=<7pt>{\scriptstyle c}*\frm{o}="2";
				"1" **\dir{-},
				(0,6)*[o]+=<7pt>{\scriptstyle d}*\frm{o}="3";
				"2" **\dir{-},
				(0,10)*[o]+=<7pt>{\scriptstyle e}*\frm{o};
				"3" **\dir{-}
			\end{xy}\end{aligned} 
	$$
\end{example}

A Lie algebra $L$ has universal enveloping preLie algebra 
$U_PL = \R L/\!\!\!\sim$ where $[x,y]\sim x\triangleleft y - y\triangleleft x$.
The preLie polynomial map $p_p:L\to U_PL$ is the composition 
$L\rightarrowtail \R L\twoheadrightarrow U_PL$.  Note again that 
the map $p_p$ is an algebra map only after twisting the preLie product to be
anti-commutative: $\mu(x,y) = x\triangleleft y - y\triangleleft x$.
It follows from adjointness properties that $U_P\Lie V \cong \R V$. 
We are interested in the map $p_p:\Lie V \to U_P\Lie V \cong \R V$ and its dual.

\begin{example}\label{E:prelie polynomials}
	Below is $p_p(\ell)$ for two simple Lie bracket expressions.
	\begin{itemize}
		\item 
			$[a,b]
		\longmapsto
			\begin{aligned}\begin{xy}
				(0,-2)*[o]+=<7pt>{\scriptstyle a}*\frm{o};
				(0,2)*[o]+=<7pt>{\scriptstyle b}*\frm{o}		
				**\dir{-}
			\end{xy}\end{aligned} -
			\begin{aligned}\begin{xy}
				(0,-2)*[o]+=<7pt>{\scriptstyle b}*\frm{o};
				(0,2)*[o]+=<7pt>{\scriptstyle a}*\frm{o}		
				**\dir{-}
			\end{xy}\end{aligned}	$

			\vskip 4pt

		\item
			$[[a,b],c]
			\longmapsto \ 
			\begin{aligned}\begin{xy}
				(0,0)*[o]+=<7pt>{\scriptstyle a}*\frm{o}="a1";
				(0,4)*[o]+=<7pt>{\scriptstyle b}*\frm{o}="b1" **\dir{-}?<>(.5)="ab1",
				(6,4)*[o]+=<7pt>{\scriptstyle a}*\frm{o}="a2";
				(6,0)*[o]+=<7pt>{\scriptstyle b}*\frm{o}="b2" **\dir{-}?<>(.5)="ab2",
				"ab1";"ab2"**\dir{}?<>(.5)*{-},
				"b2";"b1"**\dir{}?<>(.5)*+\frm<15pt,12pt>{.e}="com1",
				(3,9)*[o]+=<7pt>{\scriptstyle c}*\frm{o}="c";
				"com1";"c"**\dir{}?<>(.7);"c" **\dir{-}
			\end{xy}\end{aligned}
				\  - \ 
			\begin{aligned}\begin{xy}
				(0,5)*[o]+=<7pt>{\scriptstyle a}*\frm{o}="a1";
				(0,9)*[o]+=<7pt>{\scriptstyle b}*\frm{o}="b1" **\dir{-}?<>(.5)="ab1",
				(6,9)*[o]+=<7pt>{\scriptstyle a}*\frm{o}="a2";
				(6,5)*[o]+=<7pt>{\scriptstyle b}*\frm{o}="b2" **\dir{-}?<>(.5)="ab2",
				"ab1";"ab2"**\dir{}?<>(.5)*{-},
				"b2";"b1"**\dir{}?<>(.5)*+\frm<15pt,12pt>{.e}="com1",
				(3,0)*[o]+=<7pt>{\scriptstyle c}*\frm{o}="c";
				"com1";"c"**\dir{}?<>(.7);"c" **\dir{-}
			\end{xy}\end{aligned} \\
				{\ }\hspace{2cm} = \ \left(
			\begin{aligned}\begin{xy}
				(0,-4)*[o]+=<7pt>{\scriptstyle a}*\frm{o}="1";
				(0,0)*[o]+=<7pt>{\scriptstyle b}*\frm{o}="2" **\dir{-},
				(0,4)*[o]+=<7pt>{\scriptstyle c}*\frm{o}="3";"2" **\dir{-}
			\end{xy}\end{aligned}
				\; + \;
			\begin{aligned}\begin{xy}
				(0,0)*[o]+=<7pt>{\scriptstyle a}*\frm{o}="1";
				(-2,4)*[o]+=<7pt>{\scriptstyle b}*\frm{o}
				**\dir{-},
				(2,4)*[o]+=<7pt>{\scriptstyle c}*\frm{o};
				"1" **\dir{-}
			\end{xy}\end{aligned}
				\; - \;
			\begin{aligned}\begin{xy}
				(0,-4)*[o]+=<7pt>{\scriptstyle b}*\frm{o}="1";
				(0,0)*[o]+=<7pt>{\scriptstyle a}*\frm{o}="2" **\dir{-},
				(0,4)*[o]+=<7pt>{\scriptstyle c}*\frm{o}="3";"2" **\dir{-}
			\end{xy}\end{aligned}
				\; - \;
			\begin{aligned}\begin{xy}
				(0,0)*[o]+=<7pt>{\scriptstyle b}*\frm{o}="1";
				(-2,4)*[o]+=<7pt>{\scriptstyle a}*\frm{o}
				**\dir{-},
				(2,4)*[o]+=<7pt>{\scriptstyle c}*\frm{o};
				"1" **\dir{-}
			\end{xy}\end{aligned}
			\right) - \left(
			\begin{aligned}\begin{xy}
				(0,-4)*[o]+=<7pt>{\scriptstyle c}*\frm{o}="1";
				(0,0)*[o]+=<7pt>{\scriptstyle a}*\frm{o}="2" **\dir{-},
				(0,4)*[o]+=<7pt>{\scriptstyle b}*\frm{o}="3";"2" **\dir{-}
			\end{xy}\end{aligned}
				\; - \;
			\begin{aligned}\begin{xy}
				(0,-4)*[o]+=<7pt>{\scriptstyle c}*\frm{o}="1";
				(0,0)*[o]+=<7pt>{\scriptstyle b}*\frm{o}="2" **\dir{-},
				(0,4)*[o]+=<7pt>{\scriptstyle a}*\frm{o}="3";"2" **\dir{-}
			\end{xy}\end{aligned}
			\right)
				$
	\end{itemize}
\end{example}

%
%

\subsection{PreLie configuration pairing}

Write $(\R V)^*$ for the vector space dual of $\R V$.  The remark about infinite dimensional
coalgebra structures in the previous section still applies.

\begin{notation}
    We write $\phi\in (\R V)^*$ for a generic functional on $\R V$. We use $r$ for 
    homogeneous elements of $\R V$ 
    (a vertex-decorated rooted tree)
    and $r^*\in (\R V)^*$ for the dual functional.  Homogeneous 
    elements of $(\R V)^*$ are of the form $r^*$ for some homogeneous 
    $r\in \R V$.  Given $v\in \mathrm{Vertices}(r)$ a decorated vertex of a tree 
    $r\in \R V$,
    we abuse notation and write $v$ also for the element of the vector space $V$ decorating 
    the vertex (and $v^*$ for the dual functional of the decorating vector space element).
\end{notation}

\begin{definition}\label{D:prelie conf}
	Define the vector space pairing $\langle-,-\rangle:(\R V)^*\otimes \Lie V \to k$ to be 
	$\langle \phi,\,\ell\rangle = \phi\bigl(p_p(\ell)\bigr)$. 

	Let $\eta_p:(\R V)^* \to (\Lie V)^*$ be the associated map $\phi \mapsto \langle \phi,\, - \rangle$. 
\end{definition}

\begin{remark}
	The map $\eta_p$ is the dual of $p_p$ as a map of vector spaces.

	Given a chosen basis $B$ of $V$, the preLie polynomial of a Lie bracket 
    has the form
    $p_p(\ell) = \sum_i c_i\,r_i$
	where the $r_i$ are rooted trees with vertices decorated from $B$.  
    In this case $\langle r_i^*,\,\ell\rangle = c_i$, the coefficient of $r_i$ in 
    the preLie polynomial of $\ell$, and $\eta_p(r_i^*)\in (\Lie V)^*$ is the functional which 
    reads off the $r_i$ coefficient of preLie polynomials.
\end{remark}

Duality induces a coalgebra structure on $(\R V)^*$; taking trees to  
sums, cutting each edge of a tree in turn to
divide it into two trees and writing (root tree)$\otimes$(branch tree).  
As before, we define the cobracket to be the anti-commutative
twist\ 
$$\displaystyle
	]r^*[\ = \sum_{e\in \mathrm{E}(r)} 
	(r^{\hat e}_1)^* \otimes (r^{\hat e}_2)^* - 
	(r^{\hat e}_2)^* \otimes (r^{\hat e}_1)^*,$$
where $\sum_e$ is a sum over all edges of $r$ and $r^{\hat e}_1,r^{\hat e}_2$ are the 
rooted trees obtained by removing edge $e$, 
indexed so that $r^{\hat e}_1$ is the subtree containing the root of $r$; 
the root of $r^{\hat e}_2$ is the vertex formerly incident to $e$.

The following is proven as in Section~\ref{S:ass}.  

\begin{proposition}\label{P:prelie compat}
	If $\phi\in (\R V)^*$ and $\ell_1, \ell_2 \in \Lie V$ then
	$$\bigl\langle \phi,\, [\ell_1, \ell_2]\bigr\rangle \ = \ 
	\sum_{i} \bigl\langle \alpha_i,\, \ell_1\bigr\rangle\,
	 \bigl\langle \beta_i,\,\ell_2\bigr\rangle,$$
	where $]\phi[\ = \sum_i \alpha_i\otimes \beta_i$.
\end{proposition}

There is an alternate, combinatorial definition 
of $\langle-,-\rangle$ 
    coming from the following two observations.
	Let $r^*\in(\R V)^*$ and $\ell\in \Lie V$ be homogeneous elements written in 
	terms of a chosen basis $B$ of $V$.
	Applying Proposition~\ref{P:prelie compat},  
	if $\langle r^*,\, \ell\rangle \neq 0$ then the operation 
    of splitting apart the outer bracket $\ell = [\ell_1,\, \ell_2]$ 
    corresponds splitting $r^*$ into two trees by removing an edge.
	Furthermore, iterating Proposition~\ref{P:prelie compat} reduces
	$\langle r^*,\, \ell\rangle$ to a sum of products of $\langle b_i^*,\,b_j\rangle$
	where the $b_i^*$ are vertex decorations of $r^*$ and the $b_j$ are irreducible elements 
    bracketed together by $\ell$. 
	Thus, for $\langle r^*,\, \ell\rangle \neq 0$ there must be a label-preserving 
	bijection between the vertices of 
	$r$ and the elements bracketed together by $\ell$.   

\begin{definition}\label{D:prelie sigma conf}
 Let $r^*\in(\R V)^*$ and $\ell\in \Lie V$ be homogeneous elements. 
 A bijection 
 $\sigma:r\leftrightarrow \ell$ is 
 a bijection from the vertices of $r$ to 
 the positions in the Lie bracket expression $\ell$
 such that $v^*(\sigma(v)) \neq 0$ for all $v \in \mathrm{Vertices}(r)$.
 The weight of a bijection is $|\sigma| = \prod_v v^*(\sigma(v))$.

 Each bijection $\sigma:r\leftrightarrow\ell$ induces a map
 $\beta_\sigma: \text{Edges}(r) \to \text{Subbrackets}(\ell)$ 
	by setting 
	$\beta_\sigma\left( 
	\begin{aligned}\begin{xy}
				(0,-2)*+=<7pt>{\scriptstyle a}*\frm{o};
				(0,2)*+=<7pt>{\scriptstyle b}*\frm{o}		
				**\dir{-}
	\end{xy}\end{aligned}\right) = \text{lcb}\bigl(\sigma(a),\sigma(b)\bigr)$,
    where $\text{lcb}\bigl(\sigma(a),\sigma(b)\bigr)$ is the smallest subbracket
    expression of $\ell$ containing $\sigma(a)$ and $\sigma(b)$.
	The $\sigma$-configuration pairing of 
	$r^*$ and $\ell$ is given by the following.
	\begin{equation}\label{E:sigma conf}
		\langle r^*,\, \ell\rangle_\sigma =  |\sigma| \cdot
		\begin{cases}
			\displaystyle \prod_{e\in \mathrm{E}(r)} \text{sgn}\bigl(\beta_\sigma(e)\bigr) &
				\text{if $\beta_\sigma$ is bijective} \\
				\quad 0 & \text{otherwise}
		\end{cases}
	\end{equation}
	where $\prod_e$ is a product over all edges of $r$ and 
	$\text{sgn}\left(\beta_\sigma\left(
	\begin{aligned}\begin{xy}
					(0,-2)*+=<7pt>{\scriptstyle a}*\frm{o};
					(0,2)*+=<7pt>{\scriptstyle b}*\frm{o}		
					**\dir{-}
	\end{xy}\end{aligned}
	\right)\right) = \pm 1$ depending on whether the element $\sigma(a)$ is left or
	right of $\sigma(b)$ in the bracket expression $\ell$.\footnote{The placement of 
		$\begin{xy}(0,0)*+=<7pt>{\scriptstyle a}*\frm{o}\end{xy}$ below 
		$\begin{xy}(0,0)*+=<7pt>{\scriptstyle b}*\frm{o}\end{xy}$
		is intended to indicate that 
		$\begin{xy}(0,0)*+=<7pt>{\scriptstyle a}*\frm{o}\end{xy}$ is the vertex closer to the root.}
\end{definition}

\begin{remark}
   If everything is written in terms of a chosen basis $B$ of $V$, then 
   bijections $\sigma:r\leftrightarrow \ell$ satisfy the simpler requirement 
   $\sigma(v) = v$ for all $v\in \mathrm{Vertices}(r)$
   (vertices of $r$ are identified with positions in $\ell$ with the same basis element). 
   In this case, all bijections have weight 1.
\end{remark}

\begin{example}
   In the below examples $\sigma_i$ is the unique bijection 
   between the given rooted tree $r$ and Lie brackets $\ell$
   ($a,b,c\in V$ are linearly independent). 
\[
  \sigma_1:\begin{aligned}\begin{xy} 
	(0,0)*+=<8pt>{\scriptstyle a}*\frm{o}="1";
	(-3,6)*+=<8pt>{\scriptstyle b}*\frm{o}
	**\dir{-} ?(.4)*!R{\scriptstyle e_1},
	(3,6)*+=<8pt>{\scriptstyle c}*\frm{o};
	"1" **\dir{-} ?(.6)*!L{\,\scriptstyle e_2}
 \end{xy}\end{aligned}\  \longleftrightarrow \ 
    \bigl[c,[a,b]]
  \qquad \begin{cases} \beta_{\sigma_1}(e_1) = [a,b],  
                          & \mathrm{sgn}\bigl(\beta_{\sigma_1}(e_1)\bigr) = +1 \\
                       \beta_{\sigma_1}(e_2) = \bigl[c,[a,b]\bigr],
                          & \mathrm{sgn}\bigl(\beta_{\sigma_1}(e_2)\bigr) = -1
         \end{cases} 
\]
\[
  \sigma_2:\begin{aligned}\begin{xy} 
	(0,0)*+=<8pt>{\scriptstyle a}*\frm{o}="1";
	(-3,6)*+=<8pt>{\scriptstyle b}*\frm{o}
	**\dir{-} ?(.4)*!R{\scriptstyle e_1},
	(3,6)*+=<8pt>{\scriptstyle c}*\frm{o};
	"1" **\dir{-} ?(.6)*!L{\,\scriptstyle e_2}
 \end{xy}\end{aligned}\  \longleftrightarrow \  
     \bigl[a,[b,c]\bigr]
  \qquad \begin{cases} \beta_{\sigma_2}(e_1) = \bigl[a,[b,c]\bigr],  
                          & \mathrm{sgn}\bigl(\beta_{\sigma_2}(e_1)\bigr) = +1 \\
                       \beta_{\sigma_2}(e_2) = \bigl[a,[b,c]\bigr],
                          & \mathrm{sgn}\bigl(\beta_{\sigma_2}(e_2)\bigr) = +1
         \end{cases}
\]
  The $\sigma$-configuration pairing associated to $\sigma_1$ above is $-1$.
  The $\sigma$-configuration pairing associated to $\sigma_2$ is $0$ (because
  $\beta_{\sigma_2}$ is not bijective onto the set of subbrackets).
\end{example}
 
The $\sigma$-configuration pairing satisfies a bracket/cobracket compatibility condition 
analogous to Proposition~\ref{P:prelie compat}.
Given a bijection \(\sigma:r\leftrightarrow \ell = [\ell_1, \ell_2]\) such that
\(\langle r^*,\, \ell\rangle_\sigma \neq 0\),
let \(e\in\mathrm{E}(r)\) be the edge of \(r\)
with \(\beta_\sigma(e)=\ell\) and consider 
\(\sigma_1 = \sigma|_{r_1^{\hat{e}}}\) and \(\sigma_2 = \sigma|_{r_2^{\hat{e}}}\),
the restrictions of 
\(\sigma\) to the rooted trees \(r_1^{\hat{e}}\) and \(r_2^{\hat{e}}\)
obtained by removing edge \(e\) from \(r\).

\begin{lemma}\label{L:prelie geom conf compat}
  In the situation above 
  \[\bigl\langle r^*,\, [\ell_1,\ell_2]\bigr\rangle_\sigma = 
     \begin{cases}
       \phantom{-}
       \bigl\langle (r_1^{\hat e})^*,\,\ell_1\bigr\rangle_{\sigma_1} \cdot
        \bigl\langle (r_2^{\hat e})^*,\,\ell_2\bigr\rangle_{\sigma_2},  &
        \text{if \(\beta_\sigma(e) = +1\)} \\
       -
       \bigl\langle (r_1^{\hat e})^*,\,\ell_2\bigr\rangle_{\sigma_1} \cdot
        \bigl\langle (r_2^{\hat e})^*,\,\ell_1\bigr\rangle_{\sigma_2},  &
        \text{if \(\beta_\sigma(e) = -1\).} 
     \end{cases}\]
\end{lemma}

\begin{proof}
 It is enough to show that the restrictions \(\sigma_1\) and \(\sigma_2\)
 define bijections to the subbrackets \(\ell_1\) and \(\ell_2\) of \(\ell\).
 However, this follows from the requirement that \(\beta_\sigma \) be a bijection 
 onto the set of subbrackets of \(\ell\).  
 Since \(r_1^{\hat e}\) is connected, if \(a \in \mathrm{Vertices}(r_1^{\hat e})\) 
 is sent to a position in the 
 subbracket \(\ell_1\) while some other vertex
 \(b \in \mathrm{Vertices}(r_1^{\hat e})\)
 is sent to a position in \(\ell_2\), then there would be an edge  
 \(\epsilon\) of \(r_1^{\hat e}\)
 incident to vertices one of which is sent to a position in \(\ell_1\)
 and the other of which is sent to a position in \(\ell_2\).  In this case
 \(\beta_\sigma(\epsilon) = \ell\), so \(\beta_\sigma\) would not be injective.
\end{proof}

\begin{proposition}\label{P:prelie geom conf}
	For homogeneous $r^*\in (\R V)^*$ and $\ell\in \Lie V$,
	the configuration pairing of Definition~\ref{D:prelie conf} is equal to the following.
	\begin{equation}\label{E:prelie conf}
	\bigl\langle r^*,\, \ell\bigr\rangle \ = \!\!
		\sum_{\sigma:r \leftrightarrow \ell}
		\bigl\langle r^*,\, \ell\bigr\rangle_\sigma
	\end{equation}
	If there are no bijections $\sigma$, then 
	$\langle r^*,\,\ell\rangle = 0$.
\end{proposition}
\begin{proof}
The proposition is trivially true if $\ell$ is irreducible.    
Apply induction on the bracket length of the Lie bracket $\ell$
using Proposition~\ref{P:prelie compat} and Lemma~\ref{L:prelie geom conf compat}.
\end{proof}

%
%

\subsection{Lie coalgebras via the preLie configuration pairing.}\label{S:prelie to lie}
Once again,
let $\bigl\langle\ker(\,]\cdot[\,)\bigr\rangle\,\subset\,(\R V)^*$ be the smallest 
coideal of $(\R V)^*$ containing $\ker(\,]\cdot[\,)\setminus V^*$.  
Proposition~\ref{P:prelie compat} has the following corollary. [In the 
appendix we show that the preLie polynomial map $p_p$ is injective, which
dualizes to $\eta_p$ surjective.]

\begin{corollary}\label{C:prelie surj}
	The surjection $\eta_p:(\R V)^* \to (\Lie V)^*$ is a coalgebra homomorphism.
    Thus $(\Lie V)^* \ \cong \ (\R V)^*\,/\ \ker(\eta_P)$.
\end{corollary}

We postpone the
proof of the following until after Proposition~\ref{P:graphs ker =}. 

\begin{proposition}\label{P:prelie ker =}
	$\ker(\eta_p) = \bigl\langle\ker(\,]\cdot[\,)\bigr\rangle$.  
\end{proposition}

\begin{corollary}\label{C:prelie L^*}
	$(\Lie V)^* \ \cong \ (\R V)^*\,/\;\bigl\langle\ker(\,]\cdot[\,)\bigr\rangle$ as coalgebras.
\end{corollary}

Proposition~\ref{P:prelie ker =} implies 
Proposition~\ref{P:ass ker =} (the analogous proposition for $(\T V)^*$) 
in the following manner.

\begin{definition}\label{D:i_p}
    Let $q_p:\R V \twoheadrightarrow \T V$ be the algebra homomorphism given by 
	$q_p\bigl(a_1\triangleleft (a_2\triangleleft\cdots\triangleleft(a_{n-1}\triangleleft a_n))\bigr) = 
	 a_1a_2\cdots a_n$ and 
	$q_p(r) = 0$ for rooted trees $r$ not of this form.
	Write $i_p:(\T V)^* \rightarrowtail (\R V)^*$ for the dual of $q_p$. 
\end{definition}

Recall that the rooted tree
$a_1\triangleleft (a_2\triangleleft\cdots\triangleleft(a_{n-1}\triangleleft a_n))$
is the ``branchless tree'' with $a_1$ at the root, $a_2$ above $a_1$, $a_3$ above $a_2$, 
etc.
The map $i_p$ converts words to branchless tress.  Note that this is also a 
coalgebra homomorphism for the 
cobracket coalgebra structures of $(\T V)^*$ and $(\R V)^*$.

\begin{figure}[h]
\[ 
\begin{aligned}
\begin{xy}
	(0,0)*+<15pt,15pt>{\Lie V}="L",
	(15,10)*+{\R V}="P",
	(30,0)*+UR{\T V}="T",
	"L";"T"**\dir{-}?>*\dir{>}?<*\dir{>}?<>(.5)+<0pt,-5pt>*{p_A},
	"L";"P"**\dir{-}?>*\dir{>}?<*\dir{>}?<>(.3)+<-1pt,5pt>*{p_p},
	"P";"T"**\dir{-}?>*\dir{>>}?<>(.5)+<1pt,5pt>*{q_p},
\end{xy}
\end{aligned}
\ \longleftrightarrow \ 
\begin{aligned}
\begin{xy}
	(0,0)*+UR{(\Lie V)^*}="L",
	(15,10)*+{(\R V)^*\hspace{-10pt}}="P",
	(30,0)*+<10pt,10pt>{(\T V)^*}="T",
	"T";"L"**\dir{-}?>*\dir{>>}?<>(.5)+<0pt,-5pt>*{\eta_A},
	"P";"L"**\dir{-}?>*\dir{>>}?<>(.5)+<-1pt,5pt>*{\eta_p},
	"T";"P"**\dir{-}?>*\dir{>}?<*\dir{>}?<>(.2)+<1pt,6pt>*{i_p},
\end{xy}
\end{aligned}
\]
\end{figure}

\begin{proposition}\label{P:factor eta_A}
	$\eta_A = \eta_p \circ i_p$.  
\end{proposition}
\begin{proof}
	The proposition is trivially true on elements of $V^*\subset (\T V)^*$, where $i_p$ is merely
	the identity map.
	Using strong induction on word length, applying Propositions~\ref{P:ass compat} and 
	\ref{P:prelie compat}, we get the result for all homogeneous elements of $(\T V)^*$.
	This implies the proposition for all of $(\T V)^*$.
\end{proof}

\begin{remark}
	The previous proposition is the dual of the statement
	$p_A = q_p \circ p_p$.
\end{remark}

\begin{corollary}\label{C:factor eta_A}
	Let $\psi \in (\T V)^*$. Then $\psi\in \ker(\eta_A)$ if and only if $i_p(\psi) \in \ker(\eta_p)$.
\end{corollary}

\begin{proof}[Proof of Proposition~\ref{P:ass ker =} assuming \ref{P:prelie ker =}]
    By Corollary~\ref{C:factor eta_A}, $\ker(\eta_A) = i_p^{-1}\bigl(\ker(\eta_p)\bigr)$.  

   Write $\,]\cdot[\,_{A}$ and 
   $\,]\cdot[\,_{p}$ for the cobrackets  
   in $(\T V)^*$ and $(\R V)^*$ respectively.  
   The proof is completed by showing that 
   $\bigl\langle\ker(\,]\cdot[\,_A)\bigr\rangle
    = i_p^{-1}\Bigl(
        \bigl\langle\ker(\,]\cdot[\,_p)\bigr\rangle\Bigr)$.
   However, this is a basic property of injections of coalgebras:
   If $i:A \rightarrowtail B$ is an injection of coalgebras and $S\subset B$ then
   $\langle i^{-1}S\rangle_A = i^{-1}\langle S \rangle_B$ (see Appendix~\ref{A:coalg}).

\end{proof}

Proposition~\ref{P:prelie geom conf} combined with Proposition~\ref{P:factor eta_A}  
gives an alternative to the recursive method of Proposition~\ref{P:ass compat} 
(used in Example~\ref{E:comp}) for the calculation of coefficients in the associative 
polynomial for Lie bracket. 
\begin{align}\begin{split}
\langle \omega^*,\, \ell\rangle \ 
  &= \ \bigl\langle i_p(\omega^*),\, \ell\bigr\rangle  \\  
  &= \ \sum_{\sigma} \bigl\langle i_p(\omega^*),\, \ell\bigr\rangle_\sigma
\end{split}\label{E:convert}\end{align}

\begin{example}\label{E:alt comp}
	Recall that a right-normed Lie bracket expression has the form 
	$\ell = [a_1, [a_2, [\dots, [a_{n-1}, a_n] ] ] ]$.
	Given a word $\omega = b_1\cdots b_n$ we apply Equation~(\ref{E:convert}) to 
    compute the coefficient of $\omega$ in the associative polynomial of $\ell$.
    Nonzero bijections $\sigma:i_p(\omega)\leftrightarrow \ell$ are given
    by permutations $\sigma\in\Sigma_n$ where $b_k = a_{\sigma(k)}$.
    The map $\beta_\sigma$ 
    sends the edge between $b_k$ and $b_{k+1}$ to the subbracket 
    $\mathrm{lcb}(a_{\sigma(k)},\,a_{\sigma(k+1)})$. 
    Since $\ell$ is right-normed, least common brackets are given by  
    $\mathrm{lcb}(a_i,a_j) = [a_{\mathrm{min}(i,j)},[....]]$
    (the right-normed subbracket beginning with $a_{\mathrm{min}(i,j)}$).
    Thus if there is any $k$ such that 
    $\sigma(k-1) > \sigma(k) < \sigma(k+1)$
    then 
    $\mathrm{lcb}(a_{\sigma(k-1)},\,a_{\sigma(k)}) = 
       \mathrm{lcb}(a_{\sigma(k)},\,a_{\sigma(k+1)})$.
    In this case $\bigl\langle i_p(\omega^*),\,\ell\bigr\rangle_\sigma = 0$.
    Otherwise, $\sigma$ must be increasing until some position $k$ and 
    then be decreasing. [Note: equivalently $\sigma^{-1}$ is 
    a shuffle of $\{1,\dots,(k-1)\}$ into $\{(k+1),\dots,n\}$.]
    In this case $\bigl\langle i_p(\omega^*),\,\ell\bigr\rangle_\sigma = (-1)^{n-k}$.
    
	In other words, reading $\omega$ left-to-right should read $\ell$ moving left-to-right
	skipping some letters and then should read the remaining, skipped letters 
    from right-to-left.  
	The sign comes from the number of times you move right-to-left as in the examples below.

	\vskip 6pt

	\begin{itemize}
		\item $\bigl\langle abcdef^*,\, [a, [f, [b, [e, [c, d] ] ] ] ] \bigr\rangle = 1$ \hfill 
			$\xymatrix@C=5pt{[a, \ar@/^12pt/[rr] & [f, & [b, \ar@/^12pt/[rr] & 
				[e, \ar@/^12pt/[ll] & [c, \ar@<-3pt>@/^15pt/[r] & d]]]]] \ar@/^12pt/[ll] }$

		\vskip 13pt

		\item $\bigl\langle abcdef^*,\, [f, [a, [e, [b, [d, c] ] ] ] ] \bigr\rangle = -1$ \hfill 
			$\xymatrix@C=5pt{[f, & [a, \ar@/^12pt/[rr] & [e, \ar@/^12pt/[ll] & 
				[b, \ar@/^12pt/[rr] & [d, \ar@/^12pt/[ll] & c]]]]] \ar@<-3pt>@/^15pt/[l] }$

		\vskip 13pt

		\item $\bigl\langle abcdef^*,\, [f, [e, [a, [c, [b, d] ] ] ] ] \bigr\rangle = 0$ \hfill 
			$\xymatrix@C=5pt{[f, & [e, \ar@<-3pt>@/^15pt/[l] & [a, \ar@/^12pt/[rr] & 
			[c, \ar@/^12pt/[rr] & [b, \ar@<-3pt>@/^15pt/[l] & d]]]]] \ar@/^14pt/[llll] }$

		\vskip 13pt
		
		\item $\bigl\langle abbab^*,\, [a, [b, [b, [b, a] ] ] ]\bigr\rangle = -3$  \hfill
			$\xymatrix@C=5pt{[a, \ar@<-3pt>@/^15pt/[r] & [b, \ar@<-3pt>@/^15pt/[r] & 
				[b, \ar@/^12pt/[rr] & [b, & a]]]] \ar@<-3pt>@/^15pt/[l] }$ 
                
                \vspace{8pt}
                
			\hfill $\xymatrix@C=5pt{[a, \ar@<-3pt>@/^15pt/[r] & [b, \ar@/^12pt/[rr] & 
				[b, & [b, \ar@<-3pt>@/^15pt/[r] & a]]]] \ar@/^12pt/[ll] }$
                
                \vspace{8pt}
                
			\hfill $\xymatrix@C=5pt{[a, \ar@/^12pt/[rr] & [b, & 
				[b, \ar@<-3pt>@/^15pt/[r] & [b, \ar@<-3pt>@/^15pt/[r] & a]]]] \ar@/^12pt/[lll] }$
                
                \vspace{8pt}

	\end{itemize}
    The paranoid reader is invited to verify that these do indeed calculate coefficients
    of the given words in the Lie polynomials for their respective Lie brackets.

	The reverse construction holds for left-normed bracket expressions.  
\end{example}

Now we describe $\ker(\eta_p)$.
Define the weight of a rooted tree to be its number of vertices.
Considering the dual of Example~\ref{E:prelie polynomials}
we can read off $\ker(\eta_p)$ in low weights.
\begin{itemize}
	\item In weight 2, $\ker(\eta_p)$ is given by replacing numbers by basis elements in the 
		following expression.
	\begin{equation}\label{E:prelie ker2}
		\begin{aligned}\begin{xy}
			(0,-2)*+=<7pt>{\scriptstyle 1}*\frm{o};
			(0,2)*+=<7pt>{\scriptstyle 2}*\frm{o}		
			**\dir{-}
		\end{xy}\end{aligned} +
		\begin{aligned}\begin{xy}
			(0,-2)*+=<7pt>{\scriptstyle 2}*\frm{o};
			(0,2)*+=<7pt>{\scriptstyle 1}*\frm{o}		
			**\dir{-}
		\end{xy}\end{aligned}	
	\end{equation}

		\vskip 4pt

	\item In weight 3, $\ker(\eta_p)$ is spanned similarly by the following.
	\begin{equation}\label{E:prelie ker3}
        \begin{aligned}\begin{xy}
			(0,-4)*[o]+=<7pt>{\scriptstyle 1}*\frm{o}="1";
			(0,0)*[o]+=<7pt>{\scriptstyle 2}*\frm{o}="2" **\dir{-},
			(0,4)*[o]+=<7pt>{\scriptstyle 3}*\frm{o}="3";"2" **\dir{-}
		\end{xy}\end{aligned}
			\, + \,
		\begin{aligned}\begin{xy}
			(0,-4)*[o]+=<7pt>{\scriptstyle 2}*\frm{o}="1";
			(0,0)*[o]+=<7pt>{\scriptstyle 3}*\frm{o}="2" **\dir{-},
			(0,4)*[o]+=<7pt>{\scriptstyle 1}*\frm{o}="3";"2" **\dir{-}
		\end{xy}\end{aligned}
			\, + \,
		\begin{aligned}\begin{xy}
			(0,-4)*[o]+=<7pt>{\scriptstyle 3}*\frm{o}="1";
			(0,0)*[o]+=<7pt>{\scriptstyle 1}*\frm{o}="2" **\dir{-},
			(0,4)*[o]+=<7pt>{\scriptstyle 2}*\frm{o}="3";"2" **\dir{-}
		\end{xy}\end{aligned}
	\qquad \text{and} \qquad
		\begin{aligned}\begin{xy}
			(0,-4)*[o]+=<7pt>{\scriptstyle 1}*\frm{o}="1";
			(0,0)*[o]+=<7pt>{\scriptstyle 2}*\frm{o}="2" **\dir{-},
			(0,4)*[o]+=<7pt>{\scriptstyle 3}*\frm{o}="3";"2" **\dir{-}
		\end{xy}\end{aligned}
			\, + \, 
		\begin{aligned}\begin{xy}
			(0,-2)*[o]+=<7pt>{\scriptstyle 2}*\frm{o}="1";
			(2,2)*[o]+=<7pt>{\scriptstyle 3}*\frm{o}="2" **\dir{-},
			(-2,2)*[o]+=<7pt>{\scriptstyle 1}*\frm{o}="4";"1" **\dir{-}
		\end{xy}\end{aligned}	
        \end{equation}
\end{itemize}
\begin{example}
Other weight 3 expressions in $\ker(\eta_p)$ come from combining these. 
For example,
	\begin{equation}\label{E:prelie ker3+}
		\begin{aligned}\begin{xy}
			(0,-4)*[o]+=<7pt>{\scriptstyle 1}*\frm{o}="1";
			(0,0)*[o]+=<7pt>{\scriptstyle 2}*\frm{o}="2" **\dir{-},
			(0,4)*[o]+=<7pt>{\scriptstyle 3}*\frm{o}="3";"2" **\dir{-}
		\end{xy}\end{aligned}
			\, - \,
		\begin{aligned}\begin{xy}
			(0,-4)*[o]+=<7pt>{\scriptstyle 3}*\frm{o}="1";
			(0,0)*[o]+=<7pt>{\scriptstyle 2}*\frm{o}="2" **\dir{-},
			(0,4)*[o]+=<7pt>{\scriptstyle 1}*\frm{o}="3";"2" **\dir{-}
		\end{xy}\end{aligned},
	\qquad
		\begin{aligned}\begin{xy}
			(0,-4)*[o]+=<7pt>{\scriptstyle 1}*\frm{o}="1";
			(0,0)*[o]+=<7pt>{\scriptstyle 2}*\frm{o}="2" **\dir{-},
			(0,4)*[o]+=<7pt>{\scriptstyle 3}*\frm{o}="3";"2" **\dir{-}
		\end{xy}\end{aligned}
			 + 
		\begin{aligned}\begin{xy}
			(0,-4)*[o]+=<7pt>{\scriptstyle 1}*\frm{o}="1";
			(0,0)*[o]+=<7pt>{\scriptstyle 3}*\frm{o}="2" **\dir{-},
			(0,4)*[o]+=<7pt>{\scriptstyle 2}*\frm{o}="3";"2" **\dir{-}
		\end{xy}\end{aligned}
			 - 
		\begin{aligned}\begin{xy}
			(0,0)*[o]+=<7pt>{\scriptstyle 1}*\frm{o}="1";
			(-2,4)*[o]+=<7pt>{\scriptstyle 2}*\frm{o}="2";
			(2,4)*[o]+=<7pt>{\scriptstyle 3}*\frm{o}="3";
			"1";"2"**\dir{-},
			"1":"3" **\dir{-}
		\end{xy}\end{aligned},
	\qquad \text{and} \qquad
		\begin{aligned}\begin{xy}
			(0,0)*[o]+=<7pt>{\scriptstyle 1}*\frm{o}="1",
			(-2,4)*[o]+=<7pt>{\scriptstyle 2}*\frm{o}="2",
			(2,4)*[o]+=<7pt>{\scriptstyle 3}*\frm{o}="3",
			"1";"2"**\dir{-},
			"1";"3" **\dir{-}
		\end{xy}\end{aligned}
			 + 
		\begin{aligned}\begin{xy}
			(0,0)*[o]+=<7pt>{\scriptstyle 2}*\frm{o}="1",
			(-2,4)*[o]+=<7pt>{\scriptstyle 1}*\frm{o}="2",
			(2,4)*[o]+=<7pt>{\scriptstyle 3}*\frm{o}="3",
			"1";"2"**\dir{-},
			"1";"3"**\dir{-}
		\end{xy}\end{aligned}
			 + 
		\begin{aligned}\begin{xy}
			(0,0)*[o]+=<7pt>{\scriptstyle 3}*\frm{o}="1",
			(-2,4)*[o]+=<7pt>{\scriptstyle 1}*\frm{o}="2",
			(2,4)*[o]+=<7pt>{\scriptstyle 2}*\frm{o}="3",
			"1";"2"**\dir{-},
			"1";"3"**\dir{-}
		\end{xy}\end{aligned}.
	\end{equation}
\end{example}

Expression (\ref{E:prelie ker2}) is an 
anti-symmetry identity.  The first expression of (\ref{E:prelie ker3})
is the Arnold identity.  The second expression of (\ref{E:prelie ker3}) is 
a change of root identity.  
The expressions in (\ref{E:prelie ker3+}) are weight 3 anti-symmetry and 
Arnold with different roots.
Note that the change of root identity in (\ref{E:prelie ker3}) can be obtained
from weight 2 anti-symmetry (\ref{E:prelie ker2}) 
by adding a new vertex
$\begin{xy}(0,0)*[o]+=<7pt>{\scriptstyle 3}*\frm{o}\end{xy}$ above each occurrence of 
$\begin{xy}(0,0)*[o]+=<7pt>{\scriptstyle 2}*\frm{o}\end{xy}$.
This observation is true in general: $\ker(\eta_p)$ is local in the sense that grafting 
rooted trees onto expressions in the kernel yields new kernel expressions.

\begin{definition}
	Given rooted trees $r$, $t$ and a chosen vertex $v$ of $r$, the grafting 
	of $t$ onto $r$ at $v$, written 
	$(r\, {}_v\!\triangleleft t)$, is the rooted tree given by adding an edge 
	from the vertex $v$ of $r$ to the root vertex of $t$.
\end{definition}

\begin{remark} 
  Grafting is like a ``partial preLie product'' operation.  Recall that the 
  preLie product $r\triangleleft t$ is given by summing over all possible ways to 
  connect the root of $t$ to a vertex of $r$.  So
  $r\triangleleft t = \sum_{v\in \mathrm{V}(r)} r\, {}_v\!\triangleleft t$.
\end{remark}

\begin{proposition}\label{P:prelie graft}
	Let $r_1,...,r_n,t\in \R V$ be homogeneous elements 
    (labeled rooted trees) with vertex decorations $\mathrm{Labels}(r_i) = R$
    and $\mathrm{Labels}(t) = T$ where $R\cup T$ is a linearly
    independent subset of $V$.
	If $r_1^* + \cdots + r_n^*\in \ker(\eta_p)$, 
    then the following grafting operations give new kernel elements.
\begin{itemize}
 \item
	If $v_i\in \text{Vertices}(r_i)$ all with the same label 
    then
	$(r_1\, {}_{v_1}\!\!\!\triangleleft t)^* + \cdots + 
	 (r_n\, {}_{v_n}\!\!\!\!\triangleleft t)^* \in \ker(\eta_p).$

 \item
	If the roots of $r_i$ all have the same label, 
    then 
	$(t\, {}_{v}\!\triangleleft r_1)^* + \cdots + 
	 (t\, {}_{v}\!\triangleleft r_n)^* \in \ker(\eta_p)$ 
    for any vertex $v$ of $t$.
 \end{itemize}
\end{proposition}
  We will give a proof of the first statement; the second is similar.
  Recall that $\sum_i r_i^* \in \ker(\eta_p)$ is equivalent to  
  $\langle \sum_i r_i^*,\, \ell\rangle = 0$ for all $\ell \in \Lie V$.
\begin{proof}
  By Proposition~\ref{P:graph geom conf}, we may compute component pairings 
  $\bigl\langle (r_i\, {}_{v_i}\!\!\!\triangleleft t)^*,\, \ell\bigr\rangle$ using 
  $\sigma$-configurations.  However, since $R\cup T$ is linearly independent,
  there is at most one nonzero bijection
  $\sigma_i:(r_i\, {}_{v_i}\!\!\!\triangleleft t) \leftrightarrow \ell$.
  Let $\sigma_i$ be this bijection (if it exists) and let $e_i$ be the edge of 
  $(r_i\, {}_{v_i}\!\!\!\triangleleft t)$ connecting
  $v_i$ and the root of $t$.  
  The $\sigma_i$-configuration pairing splits into components from the trees
  $r_i$ and $t$, and the edge $e_i$ between them.
  $$\bigl\langle (r_i\, {}_{v_i}\!\!\!\triangleleft t)^*,\, \ell\bigr\rangle_{\sigma_i} 
   = \Bigl(\prod_{\epsilon\in\mathrm{E}(r_i)} \beta_{\sigma_i}(\epsilon) \Bigr)
     \Bigl(\prod_{\epsilon\in\mathrm{E}(t)} \beta_{\sigma_i}(\epsilon) \Bigr) 
     \Bigl( \beta_{\sigma_i}(e_i) \Bigr) $$
  [For simplicity, assume everything is written in terms of a basis of $V$ 
  so that all bijections have weight 1.]
  Since $T$ is linearly independent, all nontrivial bijections $\sigma_i$ must 
  act identically on the tree $t$.  Similarly, 
  since all $v_i$ have the same label, 
  the $\beta_{\sigma_i}(e_i)$ are all equal.  Thus we may factor as follows.
  $$\sum_i \bigl\langle  (r_i\, {}_{v_i}\!\!\!\triangleleft t)^*, \ell\bigr\rangle_{\sigma_i}
   = \Bigl(\sum_i \prod_{\epsilon\in\mathrm{E}(r_i)} \beta_{\sigma_i}(\epsilon) \Bigr)
     \Bigl(\prod_{\epsilon\in\mathrm{E}(t)} \beta_{\sigma_1}(\epsilon) \Bigr) 
     \Bigl( \beta_{\sigma_1}(e_1) \Bigr) $$
  The proof is completed by showing that 
   $\prod_{\epsilon\in\mathrm{E}(r_i)} \beta_{\sigma_i}(\epsilon)  = 
    \langle r_i^*,\,\ell'\rangle_{\sigma'}$ for some Lie bracket $\ell'$ and
    bijection $\sigma_i':r_i\leftrightarrow \ell'$.  This implies that the 
  sum on the right hand side above is 0.
    
  The Lie bracket $\ell'$ is given by restriction of $\ell$ to $S$ as follows.  
  Consider $\ell$ as a leaf-labeled binary rooted tree.  Remove all branches
  whose leaves are not labeled from $S$.  The resulting leaf-labeled binary rooted
  tree is $\ell'$.  By construction the bijection 
  $\sigma: (r_i\, {}_{v_i}\!\!\!\triangleleft t) \leftrightarrow \ell$
  restricts to a bijection
  $\sigma':r_i \leftrightarrow \ell'$.
\end{proof}

\begin{example}
	Grafting a new vertex 
	$\begin{aligned}\begin{xy}(0,0)*[o]+=<7pt>{\scriptstyle 4}*\frm{o}\end{xy}\end{aligned}$ 
	above
	$\begin{aligned}\begin{xy}(0,0)*[o]+=<7pt>{\scriptstyle 3}*\frm{o}\end{xy}\end{aligned}$ 
	in the $\ker(\eta_p)$ elements of (\ref{E:prelie ker3}) and (\ref{E:prelie ker3+}) yields the following 
	$\ker(\eta_p)$ elements of weight 4. 
	\begin{equation}\label{E:prelie ker4.1}
        \begin{aligned}\begin{xy}
			(0,-4)*[o]+=<7pt>{\scriptstyle 1}*\frm{o}="1";
			(0,0)*[o]+=<7pt>{\scriptstyle 2}*\frm{o}="2" **\dir{-},
			(0,4)*[o]+=<7pt>{\scriptstyle 3}*\frm{o}="3";"2" **\dir{-},
			(0,8)*[o]+=<7pt>{\scriptstyle 4}*\frm{o}="4";"3" **\dir{-}
		\end{xy}\end{aligned}
			\, + \,
		\begin{aligned}\begin{xy}
			(0,-4)*[o]+=<7pt>{\scriptstyle 2}*\frm{o}="1";
			(0,0)*[o]+=<7pt>{\scriptstyle 3}*\frm{o}="2" **\dir{-},
			(2,4)*[o]+=<7pt>{\scriptstyle 1}*\frm{o}="3";"2" **\dir{-},
			(-2,4)*[o]+=<7pt>{\scriptstyle 4}*\frm{o}="4";"2" **\dir{-}
		\end{xy}\end{aligned}
			\, + \,
		\begin{aligned}\begin{xy}
			(0,-4)*[o]+=<7pt>{\scriptstyle 3}*\frm{o}="1";
			(2,0)*[o]+=<7pt>{\scriptstyle 1}*\frm{o}="2" **\dir{-},
			(2,4)*[o]+=<7pt>{\scriptstyle 2}*\frm{o}="3";"2" **\dir{-},
			(-2,0)*[o]+=<7pt>{\scriptstyle 4}*\frm{o}="4";"1" **\dir{-}
		\end{xy}\end{aligned}
		\qquad \text{and} \qquad
		\begin{aligned}\begin{xy}
			(0,-4)*[o]+=<7pt>{\scriptstyle 1}*\frm{o}="1";
			(0,0)*[o]+=<7pt>{\scriptstyle 2}*\frm{o}="2" **\dir{-},
			(0,4)*[o]+=<7pt>{\scriptstyle 3}*\frm{o}="3";"2" **\dir{-},
			(0,8)*[o]+=<7pt>{\scriptstyle 4}*\frm{o}="4";"3" **\dir{-}
		\end{xy}\end{aligned}
			\, + \, 
		\begin{aligned}\begin{xy}
			(0,0)*[o]+=<7pt>{\scriptstyle 2}*\frm{o}="1";
			(-2,4)*[o]+=<7pt>{\scriptstyle 1}*\frm{o}
			**\dir{-},
			(2,4)*[o]+=<7pt>{\scriptstyle 3}*\frm{o}="3";
			"1" **\dir{-},
			(2,8)*[o]+=<7pt>{\scriptstyle 4}*\frm{o}="4";"3" **\dir{-}
		\end{xy}\end{aligned}
	\end{equation}
	\begin{equation}\label{E:prelie ker4.2}
		\begin{aligned}\begin{xy}
			(0,-4)*[o]+=<7pt>{\scriptstyle 1}*\frm{o}="1";
			(0,0)*[o]+=<7pt>{\scriptstyle 2}*\frm{o}="2" **\dir{-},
			(0,4)*[o]+=<7pt>{\scriptstyle 3}*\frm{o}="3";"2" **\dir{-},
			(0,8)*[o]+=<7pt>{\scriptstyle 4}*\frm{o}="4";"3" **\dir{-}
		\end{xy}\end{aligned}
			\, - \,
		\begin{aligned}\begin{xy}
			(0,-4)*[o]+=<7pt>{\scriptstyle 3}*\frm{o}="1";
			(-2,0)*[o]+=<7pt>{\scriptstyle 2}*\frm{o}="2" **\dir{-},
			(2,0)*[o]+=<7pt>{\scriptstyle 4}*\frm{o}="4";"1" **\dir{-},
			(-2,4)*[o]+=<7pt>{\scriptstyle 1}*\frm{o}="3";"2" **\dir{-}
		\end{xy}\end{aligned},
	\qquad
		\begin{aligned}\begin{xy}
			(0,-4)*[o]+=<7pt>{\scriptstyle 1}*\frm{o}="1";
			(0,0)*[o]+=<7pt>{\scriptstyle 2}*\frm{o}="2" **\dir{-},
			(0,4)*[o]+=<7pt>{\scriptstyle 3}*\frm{o}="3";"2" **\dir{-},
			(0,8)*[o]+=<7pt>{\scriptstyle 4}*\frm{o}="4";"3" **\dir{-}
		\end{xy}\end{aligned}
			 + 
		\begin{aligned}\begin{xy}
			(0,-4)*[o]+=<7pt>{\scriptstyle 1}*\frm{o}="1";
			(0,0)*[o]+=<7pt>{\scriptstyle 3}*\frm{o}="2" **\dir{-},
			(-2,4)*[o]+=<7pt>{\scriptstyle 2}*\frm{o}="3";"2" **\dir{-},
			(2,4)*[o]+=<7pt>{\scriptstyle 4}*\frm{o}="4";"2" **\dir{-}
		\end{xy}\end{aligned}
			 - 
		\begin{aligned}\begin{xy}
			(0,0)*[o]+=<7pt>{\scriptstyle 1}*\frm{o}="1";
			(-2,4)*[o]+=<7pt>{\scriptstyle 2}*\frm{o}
			**\dir{-},
			(2,4)*[o]+=<7pt>{\scriptstyle 3}*\frm{o}="3";
			"1" **\dir{-},
			(2,8)*[o]+=<7pt>{\scriptstyle 4}*\frm{o}="4";
			"3" **\dir{-}
		\end{xy}\end{aligned},
	\qquad \text{and} \qquad
		\begin{aligned}\begin{xy}
			(0,0)*[o]+=<7pt>{\scriptstyle 1}*\frm{o}="1";
			(-2,4)*[o]+=<7pt>{\scriptstyle 2}*\frm{o}
			**\dir{-},
			(2,4)*[o]+=<7pt>{\scriptstyle 3}*\frm{o}="3";
			"1" **\dir{-},
			(2,8)*[o]+=<7pt>{\scriptstyle 4}*\frm{o}="4";
			"3" **\dir{-}
		\end{xy}\end{aligned}
			 + 
		\begin{aligned}\begin{xy}
			(0,0)*[o]+=<7pt>{\scriptstyle 2}*\frm{o}="1";
			(-2,4)*[o]+=<7pt>{\scriptstyle 1}*\frm{o}
			**\dir{-},
			(2,4)*[o]+=<7pt>{\scriptstyle 3}*\frm{o}="3";
			"1" **\dir{-},
			(2,8)*[o]+=<7pt>{\scriptstyle 4}*\frm{o}="4";
			"3" **\dir{-}
		\end{xy}\end{aligned}
			 + 
		\begin{aligned}\begin{xy}
			(0,0)*[o]+=<7pt>{\scriptstyle 3}*\frm{o}="1";
			(-4,4)*[o]+=<7pt>{\scriptstyle 1}*\frm{o}
			**\dir{-},
			(4,4)*[o]+=<7pt>{\scriptstyle 2}*\frm{o};
			"1" **\dir{-},
			(0,4)*[o]+=<7pt>{\scriptstyle 4}*\frm{o}="4";
			"1" **\dir{-}
		\end{xy}\end{aligned}.
	\end{equation}
	Combining the second expression in (\ref{E:prelie ker4.1}) and 
    the first expression in (\ref{E:prelie ker4.2}) (with reversed indices) 
	yields the weight 4 anti-symmetry expression.
   \begin{equation} 
   \Biggl(\,
		\begin{aligned}\begin{xy}
			(0,-4)*[o]+=<7pt>{\scriptstyle 1}*\frm{o}="1";
			(0,0)*[o]+=<7pt>{\scriptstyle 2}*\frm{o}="2" **\dir{-},
			(0,4)*[o]+=<7pt>{\scriptstyle 3}*\frm{o}="3";"2" **\dir{-},
			(0,8)*[o]+=<7pt>{\scriptstyle 4}*\frm{o}="4";"3" **\dir{-}
		\end{xy}\end{aligned}
			\, + \, 
		\begin{aligned}\begin{xy}
			(0,0)*[o]+=<7pt>{\scriptstyle 2}*\frm{o}="1";
			(-2,4)*[o]+=<7pt>{\scriptstyle 1}*\frm{o}
			**\dir{-},
			(2,4)*[o]+=<7pt>{\scriptstyle 3}*\frm{o}="3";
			"1" **\dir{-},
			(2,8)*[o]+=<7pt>{\scriptstyle 4}*\frm{o}="4";"3" **\dir{-}
		\end{xy}\end{aligned}
    \,\Biggr) \ + \ 
    \Biggl(\,
		\begin{aligned}\begin{xy}
			(0,-4)*[o]+=<7pt>{\scriptstyle 4}*\frm{o}="1";
			(0,0)*[o]+=<7pt>{\scriptstyle 3}*\frm{o}="2" **\dir{-},
			(0,4)*[o]+=<7pt>{\scriptstyle 2}*\frm{o}="3";"2" **\dir{-},
			(0,8)*[o]+=<7pt>{\scriptstyle 1}*\frm{o}="4";"3" **\dir{-}
		\end{xy}\end{aligned}
			\, - \,
		\begin{aligned}\begin{xy}
			(0,-4)*[o]+=<7pt>{\scriptstyle 2}*\frm{o}="1";
			(-2,0)*[o]+=<7pt>{\scriptstyle 3}*\frm{o}="2" **\dir{-},
			(2,0)*[o]+=<7pt>{\scriptstyle 1}*\frm{o}="4";"1" **\dir{-},
			(-2,4)*[o]+=<7pt>{\scriptstyle 4}*\frm{o}="3";"2" **\dir{-}
		\end{xy}\end{aligned}
    \,\Biggr) \ = \ 
	    \begin{aligned}\begin{xy}
			(0,-4)*[o]+=<7pt>{\scriptstyle 1}*\frm{o}="1";
			(0,0)*[o]+=<7pt>{\scriptstyle 2}*\frm{o}="2" **\dir{-},
			(0,4)*[o]+=<7pt>{\scriptstyle 3}*\frm{o}="3";"2" **\dir{-},
			(0,8)*[o]+=<7pt>{\scriptstyle 4}*\frm{o}="4";"3" **\dir{-}
		\end{xy}\end{aligned}
			\, + \,
		\begin{aligned}\begin{xy}
			(0,-4)*[o]+=<7pt>{\scriptstyle 4}*\frm{o}="1";
			(0,0)*[o]+=<7pt>{\scriptstyle 3}*\frm{o}="2" **\dir{-},
			(0,4)*[o]+=<7pt>{\scriptstyle 2}*\frm{o}="3";"2" **\dir{-},
			(0,8)*[o]+=<7pt>{\scriptstyle 1}*\frm{o}="4";"3" **\dir{-}
		\end{xy}\end{aligned}.
    \end{equation}	
\end{example}

\begin{remark}
	The kernel of $\eta_A$ does not have a local property such as this.  For example,
	$ab-ba\in\ker(\eta_A)$; however, $abc-bac\notin\ker(\eta_A)$ and 
	$abc-bca\notin\ker(\eta_A)$.  We may attach $c$ after $b$ in $ab$; but 
	we cannot attach $c$ after $b$ in $ba$ without separating $b$ and $a$.	
\end{remark}

Proposition~\ref{P:graph ker} in the next section implies the following.

\begin{proposition}\label{P:prelie ker}
	$\ker(\eta_p)$ is spanned by graftings with weight 2 anti-symmetry 
	(\ref{E:prelie ker2}) and weight 3  
	Arnold (\ref{E:prelie ker3}) expressions.
%
%
\end{proposition}

\begin{remark}
The presence of roots in our trees makes the graftings of
Proposition~\ref{P:prelie graft}, and thus our understanding 
of $\ker(\eta_p)$, more complicated.
However, from the point of view of  
$(\Lie V)^*$, roots should not play a central role.
For example combining the first weight 4
kernel expressions of (\ref{E:prelie ker4.1}) and (\ref{E:prelie ker4.2}) 
it follows that modulo $\ker(\eta_p)$ the following rooted trees are equivalent.

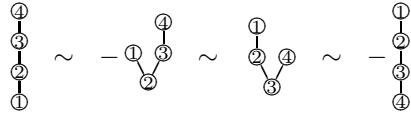
\begin{equation*}
	\begin{aligned}\begin{xy}
		(0,-4)*[o]+=<7pt>{\scriptstyle 1}*\frm{o}="1";
		(0,0)*[o]+=<7pt>{\scriptstyle 2}*\frm{o}="2" **\dir{-},
		(0,4)*[o]+=<7pt>{\scriptstyle 3}*\frm{o}="3";"2" **\dir{-},
		(0,8)*[o]+=<7pt>{\scriptstyle 4}*\frm{o}="4";"3" **\dir{-}
	\end{xy}\end{aligned}
	\ \ \sim \ \ -
	\begin{aligned}\begin{xy}
		(0,0)*[o]+=<7pt>{\scriptstyle 2}*\frm{o}="1";
		(-2,4)*[o]+=<7pt>{\scriptstyle 1}*\frm{o}
		**\dir{-},
		(2,4)*[o]+=<7pt>{\scriptstyle 3}*\frm{o}="3";
		"1" **\dir{-},
		(2,8)*[o]+=<7pt>{\scriptstyle 4}*\frm{o}="4";"3" **\dir{-}
	\end{xy}\end{aligned}
	\ \ \sim \ \  
	\begin{aligned}\begin{xy}
		(0,-4)*[o]+=<7pt>{\scriptstyle 3}*\frm{o}="1";
		(-2,0)*[o]+=<7pt>{\scriptstyle 2}*\frm{o}="2" **\dir{-},
		(2,0)*[o]+=<7pt>{\scriptstyle 4}*\frm{o}="4";"1" **\dir{-},
		(-2,4)*[o]+=<7pt>{\scriptstyle 1}*\frm{o}="3";"2" **\dir{-}
	\end{xy}\end{aligned}
	\ \ \sim \ \ - 
	\begin{aligned}\begin{xy}
		(0,-4)*[o]+=<7pt>{\scriptstyle 4}*\frm{o}="1";
		(0,0)*[o]+=<7pt>{\scriptstyle 3}*\frm{o}="2" **\dir{-},
		(0,4)*[o]+=<7pt>{\scriptstyle 2}*\frm{o}="3";"2" **\dir{-},
		(0,8)*[o]+=<7pt>{\scriptstyle 1}*\frm{o}="4";"3" **\dir{-}
	\end{xy}\end{aligned}
\end{equation*}
Grafting vertices onto the above relations yields similar relations shifting the root to 
arbitrary vertices of the weight $n$ rooted tree 
$(a_1\triangleleft(a_2\triangleleft\cdots (a_{n-1}\triangleleft a_n)))$ 
modulo $\ker(\eta_p)$.  Grafting onto these trees gives relations
moving the root to arbitrary vertices of a generic preLie tree.

In the next section
we replace rooted trees with directed graphs. 
This removes the artificial 
(from the point of view of $(\Lie V)^*$)
distinction of the root element.
\end{remark}

%
%

\section{The Configuration Pairing with Graph Coalgebras}\label{S:graphs}

\subsection{Graph algebras}

We begin by describing the graph algebra map, which takes a vector space and
makes an algebra.  The graph algebra map is the free algebra map for a
certain kind of algebra structure, but we will not elaborate on this point 
of view until the appendix.  
Instead we present graph algebras as a replacement for free preLie 
algebras.
We show that graph coalgebras contain preLie coalgebras
in the same way that the preLie coalgebras contain associative coalgebras.
Most importantly, the kernel of the map $\eta_G$ from graph coalgebras to Lie coalgebras 
has a particularly simple description.

For brevity, we say ``graph'' to mean directed, acyclic, connected, nonplanar graph.

\begin{definition}
	Let $V$ be a vector space.  
	Define $\G V$ to be
	the vector space of graphs with vertices labeled by elements of $V$, 
	modulo multilinearity. 
	The graph product 
	$g\otimes h \longmapsto \clinep{g}{h}\in \G V$ 
	is the bilinear map defined 
	on homogeneous elements as a sum over all of the ways of adding a directed edge 
	from a vertex of $g$ to a vertex of $h$, extended multilinearly to all of 
    $\G V$.
\end{definition}

\begin{example}
	Below is the graph product of the two graphs
	$
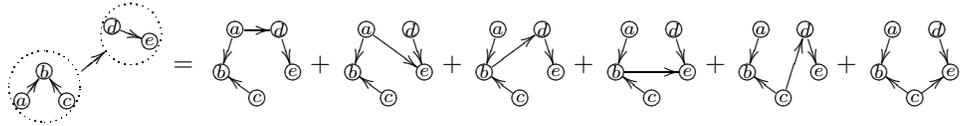
\begin{aligned}\begin{xy}
		(0,-4)*[o]+=<7pt>{\scriptstyle a}*\frm{o}="a";
		(3,0)*[o]+=<7pt>{\scriptstyle b}*\frm{o}="b" **\dir{-}?>*\dir{>},
		"b";(6,-4)*[o]+=<7pt>{\scriptstyle c}*\frm{o}="c"; **\dir{-}?>*\dir{>},
	\end{xy}\end{aligned}$ and
	$\begin{aligned}\begin{xy}
		(0,2)*[o]+=<7pt>{\scriptstyle d}*\frm{o}="d";
		(5,0)*[o]+=<7pt>{\scriptstyle e}*\frm{o}="e" **\dir{-}?>*\dir{>};
	\end{xy}\end{aligned}$.

	\begin{equation*}
		\begin{aligned}\begin{xy}
			(0,-4)*[o]+=<7pt>{\scriptstyle a}*\frm{o}="a";
			(3,0)*[o]+=<7pt>{\scriptstyle b}*\frm{o}="b" **\dir{-}?>*\dir{>},
			"b";(6,-4)*[o]+=<7pt>{\scriptstyle c}*\frm{o}="c"; **\dir{-}?>*\dir{>},
			(12,6)*[o]+=<7pt>{\scriptstyle d}*\frm{o}="d";
			(17,4)*[o]+=<7pt>{\scriptstyle e}*\frm{o}="e" **\dir{-}?>*\dir{>};
			(3,-2)*[o]+\frm<15pt,12pt>{.e},
			(14.5,5)*[o]+\frm<13pt,10pt>{.e},
			(8,0);(11,2.5) **\dir{-}?>*\dir{>},
		\end{xy}\end{aligned}
		\ = \ 
		\begin{aligned}\begin{xy}
			(-4.76,-1.55)*[o]+=<7pt>{\scriptstyle b}*\frm{o}="b";
			(0,-5)*[o]+=<7pt>{\scriptstyle c}*\frm{o}="c" **\dir{-}?<*\dir{<},
			"b";(4.76,-1.55)*[o]+=<7pt>{\scriptstyle e}*\frm{o}="e", 
			(-2.94,4.05)*[o]+=<7pt>{\scriptstyle a}*\frm{o}="a";"b" **\dir{-}?>*\dir{>},
			(2.94,4.05)*[o]+=<7pt>{\scriptstyle d}*\frm{o}="d";"e" **\dir{-}?>*\dir{>},
			"a";"d" **\dir{-}?>*\dir{>}
		\end{xy}\end{aligned}
		\ + \ 
		\begin{aligned}\begin{xy}
			(-4.76,-1.55)*[o]+=<7pt>{\scriptstyle b}*\frm{o}="b";
			(0,-5)*[o]+=<7pt>{\scriptstyle c}*\frm{o}="c" **\dir{-}?<*\dir{<},
			"b";(4.76,-1.55)*[o]+=<7pt>{\scriptstyle e}*\frm{o}="e", 
			(-2.94,4.05)*[o]+=<7pt>{\scriptstyle a}*\frm{o}="a";"b" **\dir{-}?>*\dir{>},
			(2.94,4.05)*[o]+=<7pt>{\scriptstyle d}*\frm{o}="d";"e" **\dir{-}?>*\dir{>},
			"a";"e" **\dir{-}?>*\dir{>}
		\end{xy}\end{aligned}
		\ + \ 
		\begin{aligned}\begin{xy}
			(-4.76,-1.55)*[o]+=<7pt>{\scriptstyle b}*\frm{o}="b";
			(0,-5)*[o]+=<7pt>{\scriptstyle c}*\frm{o}="c" **\dir{-}?<*\dir{<},
			"b";(4.76,-1.55)*[o]+=<7pt>{\scriptstyle e}*\frm{o}="e", 
			(-2.94,4.05)*[o]+=<7pt>{\scriptstyle a}*\frm{o}="a";"b" **\dir{-}?>*\dir{>},
			(2.94,4.05)*[o]+=<7pt>{\scriptstyle d}*\frm{o}="d";"e" **\dir{-}?>*\dir{>},
			"b";"d" **\dir{-}?>*\dir{>}
		\end{xy}\end{aligned}
		\ + \ 
		\begin{aligned}\begin{xy}
			(-4.76,-1.55)*[o]+=<7pt>{\scriptstyle b}*\frm{o}="b";
			(0,-5)*[o]+=<7pt>{\scriptstyle c}*\frm{o}="c" **\dir{-}?<*\dir{<},
			"b";(4.76,-1.55)*[o]+=<7pt>{\scriptstyle e}*\frm{o}="e",
			(-2.94,4.05)*[o]+=<7pt>{\scriptstyle a}*\frm{o}="a";"b" **\dir{-}?>*\dir{>},
			(2.94,4.05)*[o]+=<7pt>{\scriptstyle d}*\frm{o}="d";"e" **\dir{-}?>*\dir{>},
			"b";"e" **\dir{-}?>*\dir{>}
		\end{xy}\end{aligned}
		\ + \ 
		\begin{aligned}\begin{xy}
			(-4.76,-1.55)*[o]+=<7pt>{\scriptstyle b}*\frm{o}="b";
			(0,-5)*[o]+=<7pt>{\scriptstyle c}*\frm{o}="c" **\dir{-}?<*\dir{<},
			"b";(4.76,-1.55)*[o]+=<7pt>{\scriptstyle e}*\frm{o}="e",
			(-2.94,4.05)*[o]+=<7pt>{\scriptstyle a}*\frm{o}="a";"b" **\dir{-}?>*\dir{>},
			(2.94,4.05)*[o]+=<7pt>{\scriptstyle d}*\frm{o}="d";"e" **\dir{-}?>*\dir{>},
			"c";"d" **\dir{-}?>*\dir{>}
		\end{xy}\end{aligned}
		\ + \ 
		\begin{aligned}\begin{xy}
			(-4.76,-1.55)*[o]+=<7pt>{\scriptstyle b}*\frm{o}="b";
			(0,-5)*[o]+=<7pt>{\scriptstyle c}*\frm{o}="c" **\dir{-}?<*\dir{<},
			"b";(4.76,-1.55)*[o]+=<7pt>{\scriptstyle e}*\frm{o}="e",
			(-2.94,4.05)*[o]+=<7pt>{\scriptstyle a}*\frm{o}="a";"b" **\dir{-}?>*\dir{>},
			(2.94,4.05)*[o]+=<7pt>{\scriptstyle d}*\frm{o}="d";"e" **\dir{-}?>*\dir{>},
			"c";"e" **\dir{-}?>*\dir{>}
		\end{xy}\end{aligned}
	\end{equation*}
\end{example}

\vskip 5pt

By a straightforward calculation, graph products satisfy Definition~\ref{D:prelie algebra}
yielding the following.

\begin{proposition}
	$\G V$ is a preLie algebra.
\end{proposition}

\begin{corollary}
	The bracket $[x,y] = \clinep{x}{y} - \clinem{x}{y}$ makes $\G V$ a Lie algebra.
\end{corollary}

\begin{remark}
	We show in the appendix that $\G V$ has more structure than just that of a preLie 
	algebra.  Specifically it has extra, higher products which are not given
	by compositions of 
	the binary product.  In fact, preLie algebras are graph algebras whose only
	nontrivial higher products are those generated by the binary product.
\end{remark}

Since $\G V$ is a Lie algebra, there is a unique map $p_G:\Lie V \to \G V$ sending
trivial bracket
expressions to trivial graphs.  Defined recursively this map is 
$p_G\bigl([\ell_1, \ell_2]\bigr) = 
	\cllinep{p_G(\ell_1)}{p_G(\ell_2)}\ -\ \cllinem{p_G(\ell_1)}{p_G(\ell_2)}$.
In the appendix, we construct $p_G$ more generally 
via the universal enveloping graph algebra of a Lie algebra
and we show 
that $p_G:L\to U_GL$ is an injection.

\begin{example}\label{E:p_G}
	Below is $p_G(\ell)$ for two simple Lie bracket expressions.
	\begin{itemize}
		\item $[a,b]\longmapsto
			\begin{aligned}\begin{xy}
				(0,-4)*[o]+=<7pt>{\scriptstyle a}*\frm{o}="a";
				(3,0)*[o]+=<7pt>{\scriptstyle b}*\frm{o}="b" **\dir{-}?>*\dir{>},
			\end{xy}\end{aligned} 
			\ - \ 
			\begin{aligned}\begin{xy}
				(0,-4)*[o]+=<7pt>{\scriptstyle a}*\frm{o}="a";
				(3,0)*[o]+=<7pt>{\scriptstyle b}*\frm{o}="b" **\dir{-}?<*\dir{<},
			\end{xy}\end{aligned}$ 

			\vskip 5pt

		\item $[[a,b],c] \longmapsto \ 
			\begin{aligned}\begin{xy}
				(0,-4)*[o]+=<7pt>{\scriptstyle a}*\frm{o}="a1";
				(3,0)*[o]+=<7pt>{\scriptstyle b}*\frm{o}="b1" **\dir{-}?>*\dir{>},
				(9,-4)*[o]+=<7pt>{\scriptstyle a}*\frm{o}="a2";
				(12,0)*[o]+=<7pt>{\scriptstyle b}*\frm{o}="b2" **\dir{-}?<*\dir{<},
				(6,-2)*{-},
				(6,-2)*+\frm<25pt,14pt>{.e},	
				(20,2)*+=<7pt>{\scriptstyle c}*\frm{o}; (14.5,-1) **\dir{-}?<*\dir{<},
			\end{xy}\end{aligned}
			\ - \ \  
			\begin{aligned}\begin{xy}
				(0,-4)*[o]+=<7pt>{\scriptstyle a}*\frm{o}="a1";
				(3,0)*[o]+=<7pt>{\scriptstyle b}*\frm{o}="b1" **\dir{-}?>*\dir{>},
				(9,-4)*[o]+=<7pt>{\scriptstyle a}*\frm{o}="a2";
				(12,0)*[o]+=<7pt>{\scriptstyle b}*\frm{o}="b2" **\dir{-}?<*\dir{<},
				(6,-2)*{-},
				(6,-2)*+\frm<25pt,14pt>{.e},	
				(20,2)*+=<7pt>{\scriptstyle c}*\frm{o}; (14.5,-1) **\dir{-}?>*\dir{>},
			\end{xy}\end{aligned} $

			\vskip 5pt

			\qquad \qquad $ = \left[ \Bigl(
			\begin{aligned}\begin{xy}
				(0,-4)*[o]+=<7pt>{\scriptstyle a}*\frm{o}="a";
				(3,0)*[o]+=<7pt>{\scriptstyle b}*\frm{o}="b" **\dir{-}?>*\dir{>},
				(6,-4)*[o]+=<7pt>{\scriptstyle c}*\frm{o}="c",
				"a";"c" **\dir{-}?>*\dir{>},
			\end{xy}\end{aligned}
			\ + \ 
			\begin{aligned}\begin{xy}
				(0,-4)*[o]+=<7pt>{\scriptstyle a}*\frm{o}="a";
				(3,0)*[o]+=<7pt>{\scriptstyle b}*\frm{o}="b" **\dir{-}?>*\dir{>},
				(6,-4)*[o]+=<7pt>{\scriptstyle c}*\frm{o}="c",
				"b";"c" **\dir{-}?>*\dir{>},
			\end{xy}\end{aligned} \Bigr)
			\ - \  \Bigl(
			\begin{aligned}\begin{xy}
				(0,-4)*[o]+=<7pt>{\scriptstyle a}*\frm{o}="a";
				(3,0)*[o]+=<7pt>{\scriptstyle b}*\frm{o}="b" **\dir{-}?<*\dir{<},
				(6,-4)*[o]+=<7pt>{\scriptstyle c}*\frm{o}="c",
				"a";"c" **\dir{-}?>*\dir{>},
			\end{xy}\end{aligned}
			\ + \ 
			\begin{aligned}\begin{xy}
				(0,-4)*[o]+=<7pt>{\scriptstyle a}*\frm{o}="a";
				(3,0)*[o]+=<7pt>{\scriptstyle b}*\frm{o}="b" **\dir{-}?<*\dir{<},
				(6,-4)*[o]+=<7pt>{\scriptstyle c}*\frm{o}="c",
				"b";"c" **\dir{-}?>*\dir{>},
			\end{xy}\end{aligned} \Bigr)
			\right]$ 
            
            \vskip 5pt
            
            \qquad \qquad \qquad $- \ \left[ \Bigl(
			\begin{aligned}\begin{xy}
				(0,-4)*[o]+=<7pt>{\scriptstyle a}*\frm{o}="a";
				(3,0)*[o]+=<7pt>{\scriptstyle b}*\frm{o}="b" **\dir{-}?>*\dir{>},
				(6,-4)*[o]+=<7pt>{\scriptstyle c}*\frm{o}="c",
				"a";"c" **\dir{-}?<*\dir{<},
			\end{xy}\end{aligned}
			\ + \ 
			\begin{aligned}\begin{xy}
				(0,-4)*[o]+=<7pt>{\scriptstyle a}*\frm{o}="a";
				(3,0)*[o]+=<7pt>{\scriptstyle b}*\frm{o}="b" **\dir{-}?>*\dir{>},
				(6,-4)*[o]+=<7pt>{\scriptstyle c}*\frm{o}="c",
				"b";"c" **\dir{-}?<*\dir{<},
			\end{xy}\end{aligned} \Bigr)
			\ - \  \Bigl(
			\begin{aligned}\begin{xy}
				(0,-4)*[o]+=<7pt>{\scriptstyle a}*\frm{o}="a";
				(3,0)*[o]+=<7pt>{\scriptstyle b}*\frm{o}="b" **\dir{-}?<*\dir{<},
				(6,-4)*[o]+=<7pt>{\scriptstyle c}*\frm{o}="c",
				"a";"c" **\dir{-}?<*\dir{<},
			\end{xy}\end{aligned}
			\ + \ 
			\begin{aligned}\begin{xy}
				(0,-4)*[o]+=<7pt>{\scriptstyle a}*\frm{o}="a";
				(3,0)*[o]+=<7pt>{\scriptstyle b}*\frm{o}="b" **\dir{-}?<*\dir{<},
				(6,-4)*[o]+=<7pt>{\scriptstyle c}*\frm{o}="c",
				"b";"c" **\dir{-}?<*\dir{<},
			\end{xy}\end{aligned}  \Bigr)  
            \right]$
	\end{itemize}
\end{example}

%
%

\subsection{Graph configuration pairing}

As before, write $(\G V)^*$ for the vector space dual.  
Duality induces a coalgebra structure on $(\G V)^*$ which cuts a graph at all edges, writing
$(\text{source graph})\otimes(\text{target graph})$.  Define the cobracket to be the 
anti-commutative twist\ 
$]g^*[\ = \displaystyle \sum_{e\in E(g)} 
	(g^{\hat e}_1)^*\otimes (g^{\hat e}_2)^* - (g^{\hat e}_2)^*\otimes (g^{\hat e}_1)^*$,
where $\sum_e$ is a sum over the edges of $g$ and $g^{\hat e}_1$, $g^{\hat e}_2$
are the graphs obtained by removing edge $e$ which went from $g^{\hat e}_1$ to 
$g^{\hat e}_2$.  We omit the proofs below which are identical to those of 
Section~\ref{S:prelie}.

\begin{definition}\label{D:graph}
	Define the vector space pairing 
	$\langle-,-\rangle:(\G V)^*\otimes \Lie V \to k$ by
	$\langle\gamma,\,\ell\rangle = \gamma\bigl(p_G(\ell)\bigr)$.

	Let $\eta_G:(\G V)^* \to (\Lie V)^*$ be the map $\gamma\mapsto \langle\gamma,-\rangle$.
\end{definition}

\begin{remark}
	$\eta_G$ is the dual of $p_G$ as a map of vector spaces.

	If $V$ has chosen basis $B$, then the elements
	of $(\G V)^*$ are uniquely written as formal linear combinations of 
	$g^*$ where $g$ are graphs with vertex labels from $B$.
	In this case, $\langle g^*, \ell\rangle$ calculates the
	coefficient of the graph $g$ in the graph polynomial $p_G(\ell)$,
    and $\eta_G(g^*)$ is the functional which reads the $g$ coefficient of 
    graph polynomials.
\end{remark}

\begin{proposition}\label{P:graph compat}
	If $\gamma \in (\G V)^*$ then
	$\bigl\langle \gamma,\, [\ell_1,\ell_2]\bigr\rangle = 
	 \sum_i \langle \alpha_i,\, \ell_1\rangle\, \langle \beta_i,\, \ell_2\rangle$,
	where\ $]\gamma[\ = \sum_i \alpha_i\otimes \beta_i$.
\end{proposition}

\begin{definition}\label{D:sigma graph conf}
	Define bijections $\sigma:g\leftrightarrow\ell$ with induced map 
    $\beta_\sigma:\mathrm{Edges}(g)\to \mathrm{Subbrackets}(\ell)$ and 
    $\langle-,-\rangle_\sigma$ for graphs as in Definition~\ref{D:prelie sigma conf}:
\begin{align*}
	\beta_\sigma\left( 
	\linep{a}{b}
	\right) &= \text{lcb}\bigl(\sigma(a),\sigma(b)\bigr) \qquad \text{and} \qquad 
	\text{sgn}\left(\beta_\sigma\left(
	\linep{a}{b}
	\right)\right) = \pm 1. \\
	\langle g^*,\, \ell\rangle_\sigma &= 
		|\sigma| \!\! 
        \displaystyle \prod_{e\in \mathrm{E}(G)} \text{sgn}\bigl(\beta_\sigma(e)\bigr) 
			\quad \text{ if $\beta_\sigma$ is bijective.}
\end{align*}
\end{definition}

The following proposition connects to the configuration pairing of \cite{SiWa07} and \cite{SiWa09}
and is proven by induction using bracket-cobracket compatibility identical to 
Proposition~\ref{P:prelie geom conf}.

\begin{proposition}\label{P:graph geom conf}
	On homogeneous elements, the graph configuration pairing is equal to the following. 
	$$\bigl\langle g^*,\, \ell\bigr\rangle = 
	\sum_{\sigma:g\leftrightarrow\ell}
     \bigl\langle g^*,\, \ell\bigr\rangle_\sigma $$
    If there are no bijections $\sigma$, then $\langle g^*,\,\ell\rangle = 0$.
\end{proposition}

\begin{remark}
    For an example applying Proposition~\ref{P:graph geom conf} to 
    Lie algebras, see~\cite{Walt10},
    where we construct dual monomial bases for $\Lie V$ and 
    $(\G V)^* / \sim$ using the configuration pairing
    with graphs, and then we make Lie algebra computations using bracket-cobracket duality.
    Note that a short computation shows that there are no 
    dual monomial bases for $\Lie V$ and $(\T V)^*/\sim$ (using words).
\end{remark}

\begin{remark}\label{R:sigma-conf alt def}
 The requirement that $\beta_\sigma$ be bijective in Definition~\ref{D:sigma graph conf}
 can be restated in terms of subgraphs.  Given a subset of the vertices of a graph, 
 write $|S|$ for the full subgraph which they support.  Given a bijection
 $\sigma:g\leftrightarrow \ell$, it is simple to show that 
 the following are equivalent.
  \begin{itemize}
    \item $\beta_\sigma:\mathrm{Edges}(g)\to \mathrm{Subbrackets}(\ell)$ is not bijective. 
    \item $\ell$ has a subbracket $[h_1,h_2]$ such that  
      the subgraphs $|\sigma^{-1}(h_1)|$ and $|\sigma^{-1}(h_2)|$ do not have 
      exactly one edge between them in $g$.
    \item $\ell$ has a subbracket $h$ such that $|\sigma^{-1}(h)|$ is disconnected.
  \end{itemize}
 Note that having a bijection $\sigma:g\leftrightarrow \ell$ implies 
 $\#\mathrm{Edges}(g) = \#\mathrm{Subbrackets}(\ell)$.
\end{remark}

%
%

\subsection{Lie coalgebras via the graph configuration pairing}
Proposition~\ref{P:graph compat} has the following corollary.

\begin{corollary}\label{C:graph surj}
	The surjection $\eta_G:(\G V)^* \to (\Lie V)^*$ is a coalgebra homomorphism.
    Thus $(\Lie V)^* \ \cong \ (\G V)^*\,/\ \ker(\eta_G)$.
\end{corollary}

As before, define
$\bigl\langle\ker(\,]\cdot[\,)\bigr\rangle\,\subset\,(\G V)^*$ 
to be the smallest 
coideal of $(\G V)^*$ containing $\ker(\,]\cdot[\,)\setminus V^*$. 
In this section we finally prove 
$\ker(\eta_G) = \bigl\langle\ker(\,]\cdot[\,)\bigr\rangle$, which 
implies the corresponding statements in the previous sections.  
Our proof makes use of a simple, local presentation of $\ker(\eta_G)$ 
suggested at the end of the previous section.
First note that the following arrow reversing and Arnold expressions 
have cobracket $\,]\gamma[\,=0$ (and thus are also in $\ker(\eta_G)$
by Proposition~\ref{P:graph compat}).
	\begin{equation}\label{E:graph expressions}
	\begin{xy}                           
		(0,-2)*[o]+=<7pt>{\scriptstyle a}*\frm{o}="a",    
		(3,3)*[o]+=<7pt>{\scriptstyle b}*\frm{o}="b",     
		"a";"b"**\dir{-}?>*\dir{>},         
	\end{xy}\ + \
	\begin{xy}                           
		(0,-2)*[o]+=<7pt>{\scriptstyle a}*\frm{o}="a",    
		(3,3)*[o]+=<7pt>{\scriptstyle b}*\frm{o}="b",     
		"a";"b"**\dir{-}?<*\dir{<},         
	\end{xy}  
	\qquad \text{and} \qquad
	\begin{xy}                           
		(0,-2)*[o]+=<7pt>{\scriptstyle a}*\frm{o}="a",    
		(3,3)*[o]+=<7pt>{\scriptstyle b}*\frm{o}="b",   
		(6,-2)*[o]+=<7pt>{\scriptstyle c}*\frm{o}="c",   
		"a";"b"**\dir{-}?>*\dir{>},         
		"b";"c"**\dir{-}?>*\dir{>},         
	\end{xy}\ + \                             
	\begin{xy}                           
		(0,-2)*[o]+=<7pt>{\scriptstyle a}*\frm{o}="a",    
		(3,3)*[o]+=<7pt>{\scriptstyle b}*\frm{o}="b",   
		(6,-2)*[o]+=<7pt>{\scriptstyle c}*\frm{o}="c",    
		"b";"c"**\dir{-}?>*\dir{>},         
		"c";"a"**\dir{-}?>*\dir{>},          
	\end{xy}\ + \                              
	\begin{xy}                           
		(0,-2)*[o]+=<7pt>{\scriptstyle a}*\frm{o}="a",    
		(3,3)*[o]+=<7pt>{\scriptstyle b}*\frm{o}="b",   
	  (6,-2)*[o]+=<7pt>{\scriptstyle c}*\frm{o}="c",    
		"a";"b"**\dir{-}?>*\dir{>},         
		"c";"a"**\dir{-}?>*\dir{>},          
	\end{xy}                             
	\end{equation}
The following proposition appears in \cite{Sinh06.2} as Proposition~1.6 and
Theorem~1.8.  Due to its importance and simplicity, we include an outline 
of the proof.

\begin{proposition}\label{P:graph ker}
	$\ker(\eta_G)$ is spanned by local 
	arrow-reversing and Arnold expressions of graphs:
	\begin{align*}
	\text{(arrow-reversing)}\qquad & \qquad
	\begin{xy}                           
		(0,-2)*[o]+=<7pt>{\scriptstyle a}*\frm{o}="a",    
		(3,3)*[o]+=<7pt>{\scriptstyle b}*\frm{o}="b",     
		"a";"b"**\dir{-}?>*\dir{>},         
		(1.5,-5),{\ar@{. }@(l,l)(1.5,6)},
		?!{"a";"a"+/va(210)/}="a1",
		?!{"a";"a"+/va(240)/}="a2",
		?!{"a";"a"+/va(270)/}="a3",
		"a";"a1"**\dir{-},  "a";"a2"**\dir{-},  "a";"a3"**\dir{-},
		(1.5,6),{\ar@{. }@(r,r)(1.5,-5)},
		?!{"b";"b"+/va(90)/}="b1",
		?!{"b";"b"+/va(30)/}="b2",
		?!{"b";"b"+/va(60)/}="b3",
		"b";"b1"**\dir{-},  "b";"b2"**\dir{-},  "b";"b3"**\dir{-},
	\end{xy}\ + \
	\begin{xy}                           
		(0,-2)*[o]+=<7pt>{\scriptstyle a}*\frm{o}="a",    
		(3,3)*[o]+=<7pt>{\scriptstyle b}*\frm{o}="b",     
		"a";"b"**\dir{-}?<*\dir{<},         
		(1.5,-5),{\ar@{. }@(l,l)(1.5,6)},
		?!{"a";"a"+/va(210)/}="a1",
		?!{"a";"a"+/va(240)/}="a2",
		?!{"a";"a"+/va(270)/}="a3",
		"a";"a1"**\dir{-},  "a";"a2"**\dir{-},  "a";"a3"**\dir{-},
		(1.5,6),{\ar@{. }@(r,r)(1.5,-5)},
		?!{"b";"b"+/va(90)/}="b1",
		?!{"b";"b"+/va(30)/}="b2",
		?!{"b";"b"+/va(60)/}="b3",
		"b";"b1"**\dir{-},  "b";"b2"**\dir{-},  "b";"b3"**\dir{-},
	\end{xy}  \\
	\text{(Arnold)}\qquad & \qquad
	\begin{xy}                           
		(0,-2)*[o]+=<7pt>{\scriptstyle a}*\frm{o}="a",    
		(3,3)*[o]+=<7pt>{\scriptstyle b}*\frm{o}="b",   
		(6,-2)*[o]+=<7pt>{\scriptstyle c}*\frm{o}="c",   
		"a";"b"**\dir{-}?>*\dir{>},         
		"b";"c"**\dir{-}?>*\dir{>},         
		(3,-5),{\ar@{. }@(l,l)(3,6)},
		?!{"a";"a"+/va(210)/}="a1",
		?!{"a";"a"+/va(240)/}="a2",
		?!{"a";"a"+/va(270)/}="a3",
		?!{"b";"b"+/va(120)/}="b1",
		"a";"a1"**\dir{-},  "a";"a2"**\dir{-},  "a";"a3"**\dir{-},
		"b";"b1"**\dir{-}, "b";(3,6)**\dir{-},
		(3,-5),{\ar@{. }@(r,r)(3,6)},
		?!{"c";"c"+/va(-90)/}="c1",
		?!{"c";"c"+/va(-60)/}="c2",
		?!{"c";"c"+/va(-30)/}="c3",
		?!{"b";"b"+/va(60)/}="b3",
		"c";"c1"**\dir{-},  "c";"c2"**\dir{-},  "c";"c3"**\dir{-},
		"b";"b3"**\dir{-}, 
	\end{xy}\ + \                             
	\begin{xy}                           
		(0,-2)*[o]+=<7pt>{\scriptstyle a}*\frm{o}="a",    
		(3,3)*[o]+=<7pt>{\scriptstyle b}*\frm{o}="b",   
		(6,-2)*[o]+=<7pt>{\scriptstyle c}*\frm{o}="c",    
		"b";"c"**\dir{-}?>*\dir{>},         
		"c";"a"**\dir{-}?>*\dir{>},          
		(3,-5),{\ar@{. }@(l,l)(3,6)},
		?!{"a";"a"+/va(210)/}="a1",
		?!{"a";"a"+/va(240)/}="a2",
		?!{"a";"a"+/va(270)/}="a3",
		?!{"b";"b"+/va(120)/}="b1",
		"a";"a1"**\dir{-},  "a";"a2"**\dir{-},  "a";"a3"**\dir{-},
		"b";"b1"**\dir{-}, "b";(3,6)**\dir{-},
		(3,-5),{\ar@{. }@(r,r)(3,6)},
		?!{"c";"c"+/va(-90)/}="c1",
		?!{"c";"c"+/va(-60)/}="c2",
		?!{"c";"c"+/va(-30)/}="c3",
		?!{"b";"b"+/va(60)/}="b3",
		"c";"c1"**\dir{-},  "c";"c2"**\dir{-},  "c";"c3"**\dir{-},
		"b";"b3"**\dir{-}, 
	\end{xy}\ + \                              
	\begin{xy}                           
		(0,-2)*[o]+=<7pt>{\scriptstyle a}*\frm{o}="a",    
		(3,3)*[o]+=<7pt>{\scriptstyle b}*\frm{o}="b",   
	  (6,-2)*[o]+=<7pt>{\scriptstyle c}*\frm{o}="c",    
		"a";"b"**\dir{-}?>*\dir{>},         
		"c";"a"**\dir{-}?>*\dir{>},          
		(3,-5),{\ar@{. }@(l,l)(3,6)},
		?!{"a";"a"+/va(210)/}="a1",
		?!{"a";"a"+/va(240)/}="a2",
		?!{"a";"a"+/va(270)/}="a3",
		?!{"b";"b"+/va(120)/}="b1",
		"a";"a1"**\dir{-},  "a";"a2"**\dir{-},  "a";"a3"**\dir{-},
		"b";"b1"**\dir{-}, "b";(3,6)**\dir{-},
		(3,-5),{\ar@{. }@(r,r)(3,6)},
		?!{"c";"c"+/va(-90)/}="c1",
		?!{"c";"c"+/va(-60)/}="c2",
		?!{"c";"c"+/va(-30)/}="c3",
		?!{"b";"b"+/va(60)/}="b3",
		"c";"c1"**\dir{-},  "c";"c2"**\dir{-},  "c";"c3"**\dir{-},
		"b";"b3"**\dir{-}, 
	\end{xy}                             
	\end{align*}
	where 
	$\begin{aligned}\begin{xy}(0,0)*+=<7pt>{\scriptstyle a}*\frm{o}\end{xy}\end{aligned}$, 
	$\begin{aligned}\begin{xy}(0,0)*+=<7pt>{\scriptstyle b}*\frm{o}\end{xy}\end{aligned}$, 
	and 
	$\begin{aligned}\begin{xy}(0,0)*+=<7pt>{\scriptstyle c}*\frm{o}\end{xy}\end{aligned}$
	are vertices in graphs which are identical outside of the indicated area. 
\end{proposition}
\begin{proof}[Proof Sketch]
	By a computation using Proposition~\ref{P:graph geom conf} 
    similar to the proof of Proposition~\ref{P:prelie graft}, 
	the above expressions are in $\ker(\eta_G)$.

	To show that these span the entire kernel, note that modulo local arrow-reversing
	and Arnold relations, all graphs are linear combinations of ``long'' graphs, of the form
	$\begin{aligned}\begin{xy}
		(0,0)*[o]+=<10pt>{\scriptstyle b_1 }*\frm{o}="1";
		(5,3)*[o]+=<7pt>{\scriptstyle \ }*\frm{o}="2" **\dir{-}?>*\dir{>},
		"2";(10,0)*[o]+=<7pt>{\scriptstyle \ }*\frm{o}="3" **\dir{-}?>*\dir{>},
		"3";(15,3)*{\scriptstyle \cdots}="4" **\dir{-}?>*\dir{>},
		"4";(20,0)*[o]+=<7pt>{\scriptstyle \ }*\frm{o} **\dir{-}?>*\dir{>};
	\end{xy}\end{aligned}$, and recall that modulo the anti-symmetry and Jacobi 
    relations all
	Lie brackets are linear combinations of right-normed Lie bracket expressions
	$[b_1,[-,\cdots [-,-]]]$ where $b_1$ is some chosen basis element in the
    graph / bracket expression.  
	A short computation using Proposition~\ref{P:graph compat}
	shows that such ``long'' graphs pair
	perfectly under $\langle-,-\rangle$ with these right-normed Lie brackets
    (recall Example~\ref{E:alt comp}).  Since 
	right-normed Lie brackets span $\Lie V$ the local arrow-reversing and Arnold 
    relations used above
	must span $\ker(\eta_G)$ 
    (applying Corollary~\ref{C:graph surj}).
\end{proof}


The following is Propositions~3.7 and 3.18 of \cite{SiWa07}.

\begin{proposition}\label{P:graphs ker =}
	$\ker(\eta_G) = \bigl\langle\ker(\,]\cdot[\,)\bigr\rangle.$  
\end{proposition}
\begin{proof}
First note that if $\gamma \in \ker(\,]\cdot[\,)$ 
then by Proposition~\ref{P:graph compat} 
$\langle \gamma,\, \ell\rangle = 0$  for all $\ell$
(recall that $\ker(\,]\cdot[\,)$ has no weight 1 graphs).
Thus $\ker(\,]\cdot[\,) \,\subset\, \ker(\eta_G)$.
Since $\ker(\eta_G)$ is a coideal, this implies 
$\bigl\langle\ker(\,]\cdot[\,)\bigr\rangle \,\subset \,\ker(\eta_G)$.

To show $\ker(\eta_G) \subset \bigl\langle\ker(\,]\cdot[\,)\bigr\rangle$ induct on 
the number of edges connecting a local arrow-reversing or Arnold expression to the 
rest of the graph, and the number of vertices in the rest of the graph.  
The base case, with no edges connecting outside, is graphs of 
the form (\ref{E:graph expressions}).  In this case $\gamma \in \ker(\,]\cdot[\,)$ as noted
earlier.  If there is one edge connecting from the local expression to a single vertex outside then 
$\,]\gamma[\, \in \bigl(\ker(\,]\cdot[\,)\otimes(\G V)^*\bigr) \bigoplus 
                  \bigl(\G V)^*\otimes \ker(\,]\cdot[\,)\bigr)$ 
so $\gamma\in \bigl\langle\ker(\,]\cdot[\,)\bigr\rangle$.
Inducting on the number of vertices outside the local expression shows
$\gamma\in \bigl\langle\ker(\,]\cdot[\,)\bigr\rangle$ for all expressions with one edge
connecting to the outside.  Inducting also on the number of edges connecting to the outside
gives the general statement.
\end{proof}

\begin{corollary}\label{C:graphs L^*}
	$(\Lie V)^* \cong (\G V)^*\,/\;\bigl\langle\ker(\,]\cdot[\,)\bigr\rangle$ as coalgebras.
\end{corollary}

Proposition~\ref{P:prelie ker =} follows from Proposition~\ref{P:graphs ker =}
in the same manner as Proposition~\ref{P:ass ker =} in Section~\ref{S:prelie to lie}. 

\begin{definition}\label{D:i_G}
	A graph $G$ is rooted if it has a vertex $v$ such that every edge of $G$ 
	points away from $v$.  In this case, call $v$ the root of the graph $G$.

	Define $q_G:\G V \twoheadrightarrow \R V$ to be the 
	algebra homomorphism converting rooted graphs to rooted trees by forgetting
	edge directions (but remembering the root) and
	killing all non-rooted graphs. 

	Let $i_G:(\R V)^* \rightarrowtail (\G V)^*$ be the dual of $q_G$ as a vector space map. 
\end{definition}

On homogeneous elements,
$i_G(r^*)=g^*$ where $g$ is the graph obtained by orienting
each edge of the vertex-labeled, rooted tree $r$ to point away from the root.
It is clear that $i_G$ is a coalgebra homomorphism for both the standard and 
cobracket coalgebra structures on $(\R V)^*$ and $(\G V)^*$.

\begin{figure}[h]
\[ 
\begin{aligned}
\begin{xy}
	(0,0)*+<15pt,15pt>{\Lie V}="L",
	(25,8)*+{\R V}="P",
	(40,0)*+UR{\T V}="T",
    (10,16)*+{\G V}="G",
	"L";"T"**\dir{-}?>*\dir{>}?<*\dir{>}?<>(.5)+<0pt,-6pt>*{p_A},
	"L";"P"**\dir{-}?>*\dir{>}?<*\dir{>}?<>(.5)+<-1pt,5pt>*{p_p},
	"L";"G"**\dir{-}?>*\dir{>}?<*\dir{>}?<>(.5)+<-6pt,0pt>*{p_G},
	"P";"T"**\dir{-}?>*\dir{>>}?<>(.5)+<2pt,6pt>*{q_p},
	"G";"P"**\dir{-}?>*\dir{>>}?<>(.5)+<2pt,6pt>*{q_G},
\end{xy}
\end{aligned}
\ \longleftrightarrow \ 
\begin{aligned}
\begin{xy}
	(0,0)*+{(\Lie V)^*}="L",
	(25,8)*+<10pt,10pt>{(\R V)^*\hspace{-10pt}}="P",
	(40,0)*+<10pt,10pt>{(\T V)^*\hspace{-10pt}}="T",
    (10,16)*+{(\G V)^*\hspace{-10pt}}="G",
	"T";"L"**\dir{-}?>*\dir{>>}?<>(.5)+<0pt,-6pt>*{\eta_A},
	"P";"L"**\dir{-}?>*\dir{>>}?<>(.5)+<-1pt,5pt>*{\eta_p},
	"G";"L"**\dir{-}?>*\dir{>>}?<>(.4)+<-6pt,1pt>*{\eta_G},
	"T";"P"**\dir{-}?>*\dir{>}?<*\dir{>}?<>(.2)+<2pt,6pt>*{i_p},
	"P";"G"**\dir{-}?>*\dir{>}?<*\dir{>}?<>(.3)+<2pt,6pt>*{i_G},
\end{xy}
\end{aligned}
\]
\end{figure}

\begin{proposition}\label{P:factor eta_p}
	$\eta_p = \eta_G \circ i_G$.
\end{proposition}

\begin{remark}
	Proposition~\ref{P:factor eta_p} is dual to the statement $p_p = q_G \circ p_G$.
\end{remark}

\begin{corollary}\label{C:factor ker eta_p}
	Let $\phi \in (\R V)^*$.  
	Then $\phi \in \ker(\eta_p)$ if and only if $i_p(\phi) \in \ker(\eta_G)$.
\end{corollary}

\begin{proof}[Proof of Proposition~\ref{P:prelie ker =} assuming \ref{P:graphs ker =}]
	Work identical to the corresponding proof in Section~\ref{S:prelie}.
\end{proof}

Proposition~\ref{P:prelie ker} follows from Proposition~\ref{P:graph ker} using 
$i_G$ similarly.  Combining Propositions~\ref{P:factor eta_p} and \ref{P:factor eta_A},
we have the following corollaries.

\begin{corollary}\label{C:factor all}
	$\eta_A = \eta_G \circ i_G \circ i_p$.
\end{corollary}

\begin{corollary}\label{C:factor ker eta_A}
	Let $\psi \in (\T V)^*$.  
	Then $\psi \in \ker(\eta_A)$ if and only if $(i_G \circ i_p)(\psi) \in \ker(\eta_G)$.
\end{corollary}

	Applying Corollary~\ref{C:factor ker eta_A} and Proposition~\ref{P:graph ker}, 
    we may use arrow-reversing and Arnold to construct expressions in 
    $\bigl\langle\ker(\,]\cdot[\,_A)\bigr\rangle$ rather than the  
	shuffle generators.  

\begin{example}\label{E:reverse}
	$abcde - edcba \in \bigl\langle\ker(\,]\cdot[\,_A)\bigr\rangle$ because
	\[
    (i_G\circ i_p)(abcde - edcba) \ = \ 
	\begin{aligned}\begin{xy}
		(0,0)*[o]+=<9pt>{\scriptstyle a}*\frm{o}="1";
		(6,3)*[o]+=<9pt>{\scriptstyle b}*\frm{o}="2" **\dir{-}?>*\dir{>},
		"2";(12,0)*[o]+=<9pt>{\scriptstyle c}*\frm{o}="3" **\dir{-}?>*\dir{>},
		"3";(18,3)*[o]+=<9pt>{\scriptstyle d}*\frm{o}="4" **\dir{-}?>*\dir{>},
		"4";(24,0)*[o]+=<9pt>{\scriptstyle e}*\frm{o} **\dir{-}?>*\dir{>};
	\end{xy}\end{aligned}
	\ - \ 
	\begin{aligned}\begin{xy}
		(0,0)*[o]+=<9pt>{\scriptstyle a}*\frm{o}="1";
		(6,3)*[o]+=<9pt>{\scriptstyle b}*\frm{o}="2" **\dir{-}?<*\dir{<},
		"2";(12,0)*[o]+=<9pt>{\scriptstyle c}*\frm{o}="3" **\dir{-}?<*\dir{<},
		"3";(18,3)*[o]+=<9pt>{\scriptstyle d}*\frm{o}="4" **\dir{-}?<*\dir{<},
		"4";(24,0)*[o]+=<9pt>{\scriptstyle e}*\frm{o} **\dir{-}?<*\dir{<};
	\end{xy}\end{aligned}
	\]	
	is just four applications of arrow-reversing.
    
   To construct the cyclic permutations such as 
   $abcd + bcda + cdab + dabc$ in \mbox{$\bigl\langle\ker(\,]\cdot[\,_A)\bigr\rangle$} 
   we introduce
   some shorthand.  
   Given a word $a_1\cdots a_{i-1}b\,a_{i+1}\cdots a_n$, write 
   $a_1 \cdots a_{i-1}\upx{c}{b}{.2}\,a_{i+1}\cdots a_n$ for the long graph
   $a_1\cdots a_{i-1}b\,a_{i+1}\cdots a_n$ with an extra
   arrow connecting $b$ to $c$ (``attaching the letter $c$ after $b$'').  
   In this shorthand,
   modulo the local Arnold identity for graphs, we have  
   \begin{equation}\label{E:insert letter} 
    a_1\!\cdots \upx{c}{b}{.3}\,d\cdots a_n \ - \ 
    a_1\!\cdots b\,\upx{c}{d}{.3}\cdots a_n
            \ \ \sim \ \
    a_1\!\cdots b\,c\,d \cdots a_n.
   \end{equation} 
   The Arnold identity itself gives the first cyclic permutation
   $abc + bca + cab$. 
   We get the next larger cyclic permutation by
   taking the difference of attaching $d$ after $c$ and attaching $d$ after $a$,
   applying (\ref{E:insert letter}) to place $d$ between $c$ and $a$.
   \[\begin{aligned}
   &\left(ab\upx{d}{c}{-.1} \ + \ b\upx{d}{c}{-.1}a \ + \ \upx{d}{c}{-.1}ab\right)
                                     \ - \ 
     \left( \upx{d}{a}{-.1}bc \ + \ bc\upx{d}{a}{-.1} \ + \  c\upx{d}{a}{-.1}b  \right)  \\
   &\qquad =  \quad ab\upx{d}{c}{-.2} 
                    \ + \ \left(b\upx{d}{c}{-.2}a \ - \ bc\upx{d}{a}{-.2}\right)
                    \ + \ \left(\upx{d}{c}{-.2}ab \ - \ c\upx{d}{a}{-.2}b\right)
                    \ - \ \upx{d}{a}{-.2}bc \\
   &\qquad \sim  \quad ab\upx{d}{c}{-.1} 
                    \ + \ \left(bcda\right)
                    \ + \ \left(cdab\right)
                    \ - \ \upx{d}{a}{-.1}bc \\
   \end{aligned}\]
   Note that $ab\upx{d}{c}{-.2} = abcd$ 
   and $-\;\upx{d}{a}{-.2}bc \sim dabc$ (reversing the arrow from $a$ to $d$). 
   
   Larger cyclic permutation expressions are constructed similarly.
\end{example}

\begin{example}
	We can also construct the shuffle generators in 
    \mbox{$\bigl\langle\ker(\,]\cdot[\,_A)\bigr\rangle$} via 
	Corollary~\ref{C:factor ker eta_A} and Proposition~\ref{P:graph ker}
	in the following manner.  Begin with the arrow-reversing expression
	\begin{equation}\label{E:shuffle1}
	\begin{aligned}\begin{xy}
		(0,0)*[o]+=<10pt>{\scriptstyle a_1}*\frm{o}="a1",
		"a1";(5.5,-4.5)*[o]+=<10pt>{\scriptstyle a_2}*\frm{o}="a2" **\dir{-}?>*\dir{>},
		"a2";(11,-9)*[o]+=<10pt>{\scriptstyle a_3}*\frm{o}="a3" **\dir{-}?>*\dir{>},
		"a3";(16.5,-13.5)*{\scriptstyle \ddots} **\dir{-}?>*\dir{>},
		"a1";(10,0)*[o]+=<10pt>{\scriptstyle b_1}*\frm{o}="b1" **\dir{-}?>*\dir{>},
		"b1";(15.5,-4.5)*[o]+=<10pt>{\scriptstyle b_2}*\frm{o}="b2" **\dir{-}?>*\dir{>},
		"b2";(21,-9)*[o]+=<10pt>{\scriptstyle b_3}*\frm{o}="b3" **\dir{-}?>*\dir{>},
		"b3";(26.5,-13.5)*{\scriptstyle \ddots} **\dir{-}?>*\dir{>},
	\end{xy}\end{aligned}
	 +  
	\begin{aligned}\begin{xy}
		(0,0)*[o]+=<10pt>{\scriptstyle b_1}*\frm{o}="a1",
		"a1";(5.5,-4.5)*[o]+=<10pt>{\scriptstyle b_2}*\frm{o}="a2" **\dir{-}?>*\dir{>},
		"a2";(11,-9)*[o]+=<10pt>{\scriptstyle b_3}*\frm{o}="a3" **\dir{-}?>*\dir{>},
		"a3";(16.5,-13.5)*{\scriptstyle \ddots} **\dir{-}?>*\dir{>},
		"a1";(10,0)*[o]+=<10pt>{\scriptstyle a_1}*\frm{o}="b1" **\dir{-}?>*\dir{>},
		"b1";(15.5,-4.5)*[o]+=<10pt>{\scriptstyle a_2}*\frm{o}="b2" **\dir{-}?>*\dir{>},
		"b2";(21,-9)*[o]+=<10pt>{\scriptstyle a_3}*\frm{o}="b3" **\dir{-}?>*\dir{>},
		"b3";(26.5,-13.5)*{\scriptstyle \ddots} **\dir{-}?>*\dir{>},
	\end{xy}\end{aligned}
	\end{equation}
	Note that the Arnold identity
	implies the following.
	\begin{equation}\label{E:shuffle2}
	\begin{aligned}\begin{xy}
		(0,0)*[o]+=<10pt>{\scriptstyle a_1}*\frm{o}="a1",
		"a1";(5.5,-4.5)*[o]+=<10pt>{\scriptstyle a_2}*\frm{o}="a2" **\dir{-}?>*\dir{>},
		"a2";(11,-9)*[o]+=<10pt>{\scriptstyle a_3}*\frm{o}="a3" **\dir{-}?>*\dir{>},
		"a3";(16.5,-13.5)*{\scriptstyle \ddots} **\dir{-}?>*\dir{>},
		"a1";(10,0)*[o]+=<10pt>{\scriptstyle b_1}*\frm{o}="b1" **\dir{-}?>*\dir{>},
		"b1";(15.5,-4.5)*[o]+=<10pt>{\scriptstyle b_2}*\frm{o}="b2" **\dir{-}?>*\dir{>},
		"b2";(21,-9)*[o]+=<10pt>{\scriptstyle b_3}*\frm{o}="b3" **\dir{-}?>*\dir{>},
		"b3";(26.5,-13.5)*{\scriptstyle \ddots} **\dir{-}?>*\dir{>},
	\end{xy}\end{aligned}
	 \sim \ \ -  
	\begin{aligned}\begin{xy}
		(-9,2)*[o]+=<10pt>{\scriptstyle a_1}*\frm{o}="a0",
		"a0";(0,0)*[o]+=<10pt>{\scriptstyle a_2}*\frm{o}="a1" **\dir{-}?<*\dir{<},
		"a1";(5.5,-4.5)*[o]+=<10pt>{\scriptstyle a_3}*\frm{o}="a2" **\dir{-}?>*\dir{>},
		"a2";(11,-9)*{\ddots} **\dir{-}?>*\dir{>},
		"a1";(10,0)*[o]+=<10pt>{\scriptstyle b_1}*\frm{o}="b1" **\dir{-}?>*\dir{>},
		"b1";(15.5,-4.5)*[o]+=<10pt>{\scriptstyle b_2}*\frm{o}="b2" **\dir{-}?>*\dir{>},
		"b2";(21,-9)*[o]+=<10pt>{\scriptstyle b_3}*\frm{o}="b3" **\dir{-}?>*\dir{>},
		"b3";(26.5,-13.5)*{\scriptstyle \ddots} **\dir{-}?>*\dir{>},
	\end{xy}\end{aligned}
	 -  
	\begin{aligned}\begin{xy}
		(-9,2)*[o]+=<10pt>{\scriptstyle a_1}*\frm{o}="a0",
		"a0";(0,0)*[o]+=<10pt>{\scriptstyle b_1}*\frm{o}="a1" **\dir{-}?<*\dir{<},
		"a1";(5.5,-4.5)*[o]+=<10pt>{\scriptstyle a_2}*\frm{o}="a2" **\dir{-}?>*\dir{>},
		"a2";(11,-9)*[o]+=<10pt>{\scriptstyle a_3}*\frm{o}="a3" **\dir{-}?>*\dir{>},
		"a3";(16.5,-13.5)*{\scriptstyle \ddots} **\dir{-}?>*\dir{>},
		"a1";(10,0)*[o]+=<10pt>{\scriptstyle b_2}*\frm{o}="b1" **\dir{-}?>*\dir{>},
		"b1";(15.5,-4.5)*[o]+=<10pt>{\scriptstyle b_3}*\frm{o}="b2" **\dir{-}?>*\dir{>},
		"b2";(21,-9)*{\scriptstyle \ddots} **\dir{-}?>*\dir{>},
	\end{xy}\end{aligned}
	\end{equation}
	Reversing the arrows to 
	$\begin{aligned}\begin{xy}(0,0)*[o]+=<9pt>{\scriptstyle a_1}*\frm{o}\end{xy}\end{aligned}$
	on the right-hand side of (\ref{E:shuffle2}) changes each sign.   
	Iterating (\ref{E:shuffle2}) beginning with the first
	term of (\ref{E:shuffle1}) yields all shuffles of $(a_1a_2\cdots)$ into $(b_1b_2\cdots)$
	with first letter $a_1$.  Iterating (\ref{E:shuffle2}) beginning with the second
	term of (\ref{E:shuffle1}) yields all shuffles with first letter $b_1$.
\end{example}

    In a sequel we work similar to above applying 
	Corollary~\ref{C:factor ker eta_A} and Proposition~\ref{P:graph ker}
    to compute
	new bases for free Lie algebras via a series of 
	combinatorial moves, analyzing and comparing computability and utility of results.

%
%

\subsection{Combinatorial associative configuration pairing}

Pulling back the graph configuration pairing from $(\G V)^*$ to $(\T V)^*$ 
using Definition~\ref{D:sigma graph conf} and Remark~\ref{R:sigma-conf alt def}  
yields an alternate 
method for computing the configuration pairing between associative coalgebras and Lie algebras.

\begin{definition}
Given a length $n$ word $\omega=(a_1a_2\cdots a_n) \in \T^n V$, 
a homogeneous length $n$ Lie bracket 
$\ell \in \Lie^n V$ bracketing the letters $b_1,\dots,b_n$, 
and a permutation $\sigma\in S_n$, write $\sigma$ also for the induced 
map $\sigma : a_i \mapsto b_{\sigma(i)}$ from letters of $\omega$ to letters of $\ell$.
The $\sigma$-configuration
pairing of $\omega^*$ and $\ell$ is given by the following.
\begin{equation}
\bigl\langle \omega^*,\, \ell \bigr\rangle_\sigma \ = \ \prod_i a_i^*(b_{\sigma(i)}) \cdot
 \begin{cases} 
        \mathrm{sgn}(\sigma)  &
           \parbox[t]{.45\textwidth}{if every subbracket $h$ of $\ell$ has  
                                    $\sigma^{-1}(h)$ a connected subword of $\omega$,} \\
               0 & \text{otherwise}
 \end{cases}
\end{equation}
where $\mathrm{sgn}(\sigma) = \pm 1$ is the usual sign of the permutation $\sigma\in S_n$.
\end{definition}

\begin{proposition}\label{P:ass combinatorial def}
  $\displaystyle \bigl\langle \omega^*,\, \ell\bigr\rangle = 
   \sum_{\sigma \in S_n} \bigl\langle \omega^*,\, \ell\bigr\rangle_\sigma$.
\end{proposition}

In contrast to the earlier Proposition~\ref{P:ass compat} used in Example~\ref{E:comp}
computing configuration pairings recursively from the outermost bracket of $\ell$ working inwards,
Proposition~\ref{P:ass combinatorial def} may be used to compute configuration pairings
inductively from the innermost brackets of $\ell$ working outwards.

\begin{example}\label{E:ass combinatorial def}
We may apply Proposition~\ref{P:ass combinatorial def} to compute configuration pairings as below.
\begin{itemize}
\item
 $\bigl\langle abcb^*,\ [[[a,c],b],b]\bigr\rangle = 0$ because the Lie bracket has subbracket
 $[a,c]$ but the letters $a$ and $c$ are not adjacent in the word $abcb$.
\item
 $\bigl\langle abbba^*,\ [[[b,a],b],[a,b]]\bigr\rangle = 2$ because the two inner subbrackets
  $[b,a]$ and $[a,b]$ must correspond under $\sigma$ to either $ab$ or $ba$ at the start or
  end of the word.  If $[b,a]$ corresponds to $ab$, then the permutation $\sigma$ is
  $(12)(3)(45)$.  If $[b,a]$ corresponds to $ba$, then the permutation $\sigma$ is
  $(14)(3)(25)$.  Both of these have $\mathrm{sgn}(\sigma) = 1$.
\end{itemize}
\end{example}

\appendix

%
%

\section{Operad Structures}\label{A:operad}

Operads are objects which encode algebra structures.  On the set-level, they consist
of an element for every possible way of combining things using the algebra structure, 
along with ``composition'' maps expressing some combinations as compositions of others.
More formally, a 
(unital, symmetric) operad $\Op$ in the symmetric monoidal category of $k$-vector 
spaces is a symmetric sequence of vector spaces, $\{\Op(n)\}_{n\ge 0}$ where 
each $\Op(n)$ has $\Sigma_n$-action, as well as 
a unit $k \to \Op(0)$ and
and equivariant composition maps,
$\Op(n)\bigotimes \Op(k_1)\otimes\cdots\otimes \Op(k_n) \mapsto 
 \Op(\sum_i k_i)$, satisfying standard unital and associativity axioms. 
The composition map tells which $(\sum_i k_i)$-ary operation is given by
combining $k_1$,\dots,$k_n$-ary operations together via an $n$-ary operation. 
The symmetric group action accounts for plugging elements into an $n$-ary operation 
in different orders.
Below we use $\circ_i$ operations to define operad structure.  These are maps
$\Op(n)\otimes \Op(m) \mapsto \Op(m+n-1)$ which plug an $m$-ary operation into an
$n$-ary operation at position $i$.

\begin{itemize}
\item
$\Aop$. The associative operad is given by $\Aop(n)\cong k[\Sigma_n]$ the regular
representations of the symmetric groups.  Composition is given by wreath 
product.

\item 
$\Lop$. The Lie operad has $\Lop(n)$ given by the $k$-vector space generated
by formal length $n$ bracket expressions of the elements $a_1,\dots,a_n$.  This is  
isomorphic to the $k$-vector space of rooted binary planar trees whose  
leaf set is $[n] = \{1,\dots,n\}$ (with $\Sigma_n$ permuting $[n]$)
modulo anti-symmetry and Jacobi identites of trees.

\item
$\Rop$. \cite{ChLi00}  The preLie operad is isomorphic to the operad of
rooted trees. 
$\Rop(n)$ is the $k$-vector space of rooted trees with vertex set  
$[n]$.
The operad structure of $\Rop(n)$ comes from the following $\circ_i$ operation.
Direct the edges of a rooted tree to point away from the root.
$R \circ_i T$ is given by replacing vertex $i$ of $R$ with the tree $T$.  The
incoming edge to $i$ (if $i$ is not the root of $R$) connects to the root of $T$, 
and we sum over all ways that the outgoing edges of $i$ can be assigned source
vertices from $T$.

\item
$\Gop$. The graph operad has $\Gop(n)$ given by the $k$-vector space of directed, acyclic graphs 
with vertex set $[n]$.
The operad structure on $\Gop$ comes from the 
following $\circ_i$ operation (generalizing that of $\Rop$).
$G\circ_i H$ is given by replacing vertex $i$ of $G$ by the graph $H$, summing over all
ways that edges in $G$ with source or target vertex $i$ can be assigned a new source or target 
vertex in $H$.

\end{itemize}

\begin{remark}
	Write $\Op^\vee$ for the arity-wise dual of $\Op$:  i.e. $\Op^\vee(n) = \Op(n)^*$.  If 
	$\Op$ is an (arity-wise finitely generated) operad, then $\Op^\vee$ is a cooperad.
	The dual cooperad structure of $\Rop^\vee$ acts by quotienting subtrees to vertices.
	The dual cooperad structure of $\Gop^\vee$ acts by quotienting subgraphs to vertices,
	as described in \cite{SiWa07}.
\end{remark}


\begin{proposition}
	$\Gop$ is not a quadratic operad \cite{GiKa94}.
\end{proposition}
\begin{proof}
	Count ranks as $k[\Sigma_n]$-modules.
	$\Gop(2)$ has rank 1 as a $k[\Sigma_2]$-module. 
	$\Gop(3)$ is spanned as a $k[\Sigma_3]$-module by
	$\begin{xy}                           
		(0,-2)*[o]+=<7pt>{\scriptstyle\ }*\frm{o}="a",    
		(3,3)*[o]+=<7pt>{\scriptstyle\ }*\frm{o}="b",   
		(6,-2)*[o]+=<7pt>{\scriptstyle\ }*\frm{o}="c",   
		"a";"b"**\dir{-}?>*\dir{>},         
		"b";"c"**\dir{-}?>*\dir{>},         
	\end{xy}$,
	$\begin{xy}                           
		(0,-2)*[o]+=<7pt>{\scriptstyle\ }*\frm{o}="a",    
		(3,3)*[o]+=<7pt>{\scriptstyle\ }*\frm{o}="b",   
		(6,-2)*[o]+=<7pt>{\scriptstyle\ }*\frm{o}="c",   
		"b";"a"**\dir{-}?>*\dir{>},         
		"b";"c"**\dir{-}?>*\dir{>},         
	\end{xy}$, and
	$\begin{xy}                           
		(0,-2)*[o]+=<7pt>{\scriptstyle\ }*\frm{o}="a",    
		(3,3)*[o]+=<7pt>{\scriptstyle\ }*\frm{o}="b",   
		(6,-2)*[o]+=<7pt>{\scriptstyle\ }*\frm{o}="c",   
		"a";"b"**\dir{-}?>*\dir{>},         
		"c";"b"**\dir{-}?>*\dir{>},         
	\end{xy}$.
	However, a quadratic operad 
	with $\Op(2)$ of rank 1 cannot have $\Op(3)$ of rank $>2$.
\end{proof}

%

There are quotient maps of operads
$\Gop \twoheadrightarrow \Rop \twoheadrightarrow \Aop$ defined as follows.
\begin{itemize}
\item
The map $Q_p:\Rop \twoheadrightarrow \Aop$ is induced by the functor which 
views an associative algebra as a preLie algebra.  
$Q_p$ takes rooted trees which are bivalent at all but
two vertices to the permutation encoded by the
vertices from the
root to the leaf, and quotients all trees containing a vertex of valency $>2$.

\item
The map $Q_G:\Gop \twoheadrightarrow \Rop$ is induced by the functor which
views a preLie algebra as a graph algebra (by viewing an operation encoded
by a rooted tree as an operation encoded by a rooted 
directed graph).
$Q_G$ takes rooted graphs to rooted trees, and quotients
non-rooted graphs.  
[The interested reader may check that this commutes with 
$\circ_i$ operations.]
\end{itemize}

There are inclusion maps of operads
$\Lop \rightarrowtail \Aop$, $\Lop \rightarrowtail \Rop$, $\Lop \rightarrowtail \Gop$
defined as follows.

\begin{itemize}
\item
The map $U_A:\Lop \rightarrowtail \Aop$ is induced by the map viewing an associative algebra
as a Lie algebra with bracket $[x,y] = xy-yx$.  This map is an injection by Poincar\'e-Birkhoff-Witt.
The operad map $U_A$ induces adjoint maps.
\begin{itemize}
\item
$F_A:(\Aop\mathrm{-\emph{algebras}}) \to (\Lop\mathrm{-\emph{algebras}})$ 
forgetting associative product structure down to Lie algebra
structure ($[x,y] = xy - yx$).
\item
$U_A:(\Lop\mathrm{-\emph{algebras}}) \to (\Aop\mathrm{-\emph{algebras}})$ 
sending a Lie algebra to its enveloping algebra. 
\end{itemize}
The unit of this adjuction is $L \mapsto F_A U_A L$ sending a Lie algebra to its
universal enveloping algebra with anti-commutative product.  Elementwise, the induced
map is $p_A$ sending a Lie bracket to its Lie polynomial in the universal enveloping algebra. 

\item
The map $U_p:\Lop \rightarrowtail \Rop$ is induced by the map viewing a preLie algebra as
a Lie algebra with bracket $[x,y] = x\triangleleft y - y\triangleleft x$.
From definitions, it follows that $U_A = Q_p U_p$.  Since $U_A$ is an injection, so is $U_P$. 
The injection $p_p$ is induced by the unit of the adjunction associated to the operad injection 
$U_p$ as before. 

\item
The map $U_G:\Lop \rightarrowtail \Gop$ is induced by the map viewing a graph algebra as a 
Lie algebra with bracket 
	$[x,y] = \clinep{x}{y} - \clinem{x}{y}$.
From definitions, it follows that $U_A = Q_p Q_G U_G$.  Since $U_A$ is an injection, so is $U_G$. 
The injection $p_G$ is induced by the unit of the adjunction associated to the operad injection 
$U_G$ as before. 
\end{itemize}

%
%

\section{Basic Coalgebra Coideal Facts}\label{A:coalg}

A noncounital coalgebra is $(C, \Delta)$ where $C$ is a vector space with coproduct 
$\Delta:C\to C\otimes C$.  
A homomorphism of coalgebras $f:(C,\Delta_C) \to (D,\Delta_D)$ is a map $f:C \to D$ such that
$(f\otimes f)\Delta_C = \Delta_D(f)$.
A coideal of a noncounital coalgebra is 
$I\subset C$ with $\Delta(I) \subset (I\otimes C)\, \bigoplus\, (C\otimes I)$.
Note that the intersection of two coideals is also a coideal.
Given a subset $S\subset C$, the coideal generated by $S$ is 
$\langle S\rangle$ the smallest coideal of $C$ containing $S$.

\begin{proposition}
 If $f:A\to B$ is a coalgebra homomorphism then $\ker(f)$ is a coideal.
\end{proposition}
\begin{proof}
 Let $a\in \ker(f)$.  Then $(f\otimes f)\Delta(a) = \Delta f(a) = \Delta 0 = 0\otimes 0$.
 Thus $\Delta a \in (\ker(f)\otimes B)\,\bigoplus\,(B\otimes\ker(f)).$
\end{proof}

\begin{proposition}\label{P:coalg f(coideal)}
 If $f:A\to B$ is a coalgebras homomorphism and $X\subset A$ is a coideal, then
 $f(X) \subset B$ is a coideal.
\end{proposition}
\begin{proof}
Suppose $x\in X$ a coideal of $A$.  Then 
 $\Delta_A x \ \subset \ (X\otimes A)\,\bigoplus\,(A\otimes X)$ 
 so
 \[
 \Delta_B f(x) \ = \ (f \otimes f) \Delta_A x \ \subset \  
    \bigl(f(X)\otimes B\bigr)\,\bigoplus\,\bigl(B\otimes f(X)\bigr).
 \]
\end{proof}

\begin{proposition}\label{P:coalg i^{-1}coideal}
 If $f:A\to B$ is a coalgebras homomorphism and $Y\subset B$ is a coideal, then
 $f^{-1}Y \subset A$ is a coideal.
\end{proposition}
\begin{proof}
 Suppose $f(x) \in Y$ a coideal of $B$.  Then  
 \[
  (f\otimes f) \Delta_A x  \ =  \ \Delta_B f(x) \ \subset \ (Y\otimes B)\,\bigoplus\, (B\otimes Y).
 \]
Thus $\Delta_A x \ \subset \ (f^{-1}Y\otimes A)\, \bigoplus\, (A\otimes f^{-1}Y)$. 
\end{proof}

\begin{proposition}\label{P:coalg inj}
 If $i:A\rightarrowtail B$ is a coalgebra injection and $S\subset B$ is a subset, then 
 $\langle i^{-1}S\rangle_A = i^{-1}\langle S\rangle_B$.
\end{proposition}
\begin{proof}
 Suppose $i:A\rightarrowtail B$ and $S\subset B$.
 Since $S\subset i\langle i^{-1}S\rangle_A$ and 
 $i\langle i^{-1}S\rangle_A$ is a coideal 
 we have 
 $i\,i^{-1}\langle S\rangle_B \ = \ \langle S\rangle_B \ \subset \ i\langle i^{-1} S \rangle_A.$
 
 Similarly, 
 since $i^{-1}S \subset i^{-1}\langle S\rangle_B$ and 
 $i^{-1}\langle S\rangle_B$ is a coideal, 
 $\langle i^{-1}S \rangle_A \subset i^{-1}\langle S \rangle_B$.
 In particular 
 $i \langle i^{-1}S \rangle_A \subset i\,i^{-1}\langle S \rangle_B$.
 
 Thus $i\,i^{-1}\langle S\rangle_B  = i\langle i^{-1} S \rangle_A$.  The map $i$
 is an injection so this implies the desired result, 
 $i^{-1}\langle S\rangle_B  = \langle i^{-1} S \rangle_A$.
\end{proof}

\begin{remark}
 The statements above are dual to the following algebra facts:
 \begin{itemize}
 \item[\textbf{\ref{P:coalg f(coideal)}*}]
   If $f:A\to B$ is an algebra homomorphism and $Y\subset B$ is an ideal
   then $f^{-1}Y \subset A$ is an ideal.
 \item[\textbf{\ref{P:coalg i^{-1}coideal}*}]
   If $f:A\to B$ is an algebra homomorphism and $X\subset A$ is an ideal
   then $f(X) \subset f(A)$ is an ideal (of the image).
 \item[\textbf{\ref{P:coalg inj}*}]
   If $f:A\twoheadrightarrow B$ is a surjection of algebras and $S\subset A$
   is a subset then $f\bigl(\langle S\rangle_A\bigr) = \bigl\langle f (S)\bigr\rangle_B$.
 \end{itemize}
\end{remark}


\end{document}